\documentclass[12pt,a4paper]{article}

\usepackage{graphicx}
\usepackage{subfig}
\usepackage{amsmath,amsthm,amssymb}
\usepackage{esint}
\usepackage{mathrsfs}
\usepackage{appendix}

\usepackage{hyperref}

\usepackage[left=2.2cm,right=2.2cm,top=3cm,bottom=3.5cm]{geometry}

\usepackage{color}
\usepackage{array}
\usepackage{url}

\usepackage{float}

\newtheorem{theorem}{Theorem}[section]
\newtheorem*{theorem*}{Theorem}
\newtheorem{lemma}[theorem]{Lemma}

\newtheorem{proposition}[theorem]{Proposition}

\theoremstyle{definition}
\newtheorem{definition}[theorem]{Definition}
\newtheorem{conjecture}[theorem]{Conjecture}

\theoremstyle{remark}
\newtheorem{remark}[theorem]{Remark}
\usepackage{authblk}

\usepackage{mathtools}
\usepackage{amssymb}
\usepackage{appendix}

\numberwithin{equation}{section}

\usepackage[numbers,sort&compress]{natbib}

\begin{document}
	
\title{Onsager's Conjecture for Subgrid Scale $\alpha$-Models of Turbulence}
	
\author[]{Daniel W. Boutros\footnote{Department of Applied Mathematics and Theoretical Physics, University of Cambridge, Cambridge CB3 0WA UK. Email: \textsf{dwb42@cam.ac.uk}} \space and Edriss S. Titi\footnote{Department of Mathematics, Texas A\&M University, College Station, TX 77843-3368, USA; Department of Applied Mathematics and Theoretical Physics, University of Cambridge, Cambridge CB3 0WA UK; also Department of Computer Science and Applied Mathematics, Weizmann Institute of Science, Rehovot 76100, Israel. Emails: \textsf{titi@math.tamu.edu} \; \textsf{Edriss.Titi@maths.cam.ac.uk} \; \textsf{edriss.titi@weizmann.ac.il}}}

\date{September 19, 2022}
	
\maketitle

{\centering \textit{In memory of Charles R. Doering} \par}
	
\bigskip

\begin{abstract}
The first half of Onsager's conjecture states that the Euler equations of an ideal incompressible fluid conserve energy if $u (\cdot ,t) \in C^{0, \theta} (\mathbb{T}^3)$ with $\theta > \frac{1}{3}$. In this paper, we prove an analogue of Onsager's conjecture for several subgrid scale $\alpha$-models of turbulence. In particular we find the required H\"older regularity of the solutions that ensures the conservation of energy-like quantities (either the $H^1 (\mathbb{T}^3)$ or $L^2 (\mathbb{T}^3)$ norms) for these models.

We establish such results for the Leray-$\alpha$ model, the Euler-$\alpha$ equations (also known as the inviscid Camassa-Holm equations or Lagrangian averaged Euler equations), the modified Leray-$\alpha$ model, the Clark-$\alpha$ model and finally the magnetohydrodynamic Leray-$\alpha$ model. In a sense, all these models are inviscid regularisations of the Euler equations; and formally converge to the Euler equations as the regularisation length scale $\alpha \rightarrow 0^+$.

Different H\"older exponents, smaller than $1/3$, are found for the regularity of solutions of these models (they are also formulated in terms of Besov and Sobolev spaces) that guarantee the conservation of the corresponding energy-like quantity. This is expected due to the smoother nonlinearity compared to the Euler equations. These results form a contrast to the universality of the $1/3$ Onsager exponent found for general systems of conservation laws by (Gwiazda et al., 2018; Bardos et al., 2019).
\end{abstract}


\vspace{0.1cm} \noindent \textbf{Mathematics Subject Classification:} 76B99 (primary), 35Q35, 35D30, 76F99 (secondary)

\newpage

\section{Introduction}

It is well-known that the $L^2$ spatial norm (the kinetic energy) of smooth solutions of the Euler equations of an ideal incompressible fluid is conserved. On the other hand, turbulent flows, for high Reynolds numbers, in particular for an infinite Reynolds number, are not expected to be smooth. Therefore, they might not conserve energy. This observation was made by Lars Onsager \cite{onsager} which led to the following conjecture (which is now a theorem).
\begin{conjecture}[Onsager's conjecture]
Let $v \in L^\infty ((0,T); L^2 (\mathbb{T}^3) )$ be a weak solution of the Euler equations.
If $v \in L^3 ((0,T); C^{0, \theta} (\mathbb{T}^3) )$ with $\theta > \frac{1}{3}$, then the spatial $L^2 (\mathbb{T}^3)$ norm of the solution stays constant.

Furthermore, for $\theta < 1/3$ there exist weak solutions of the Euler equation that do not conserve energy.
\end{conjecture}

In this context $\frac{1}{3}$ is often referred to as the Onsager exponent. The importance of Onsager's conjecture in mathematical fluid mechanics was pointed out in \cite{eyink}. The first half of Onsager's conjecture, i.e. about conservation of energy, was proved in \cite{eyink}, but with slightly stronger assumptions than $C^{0, \theta}$ with $\theta > \frac{1}{3}$. Subsequently, this part of Onsager's conjecture was proved (i.e. if $\theta > \frac{1}{3}$) in \cite{constantin} using Besov spaces and commutator estimates. A different proof of the first half was given in \cite{duchon}, which relied on a local equation of energy balance.

The paper \cite{robinson} studies the first half of the conjecture in the upper-half space $\mathbb{T}^2 \times \mathbb{R}_+$, subject to non-permeability boundary conditions on the wall/boundary $x_3 = 0$. This part of the conjecture was then proven in bounded domains with $C^2$ boundary in \cite{titi2018}. The results of \cite{titi2018} were improved in \cite{titi2019} by only requiring interior H\"older regularity (instead of on the whole domain), for this the regularity of the pressure needed to be studied more carefully (which was also done in \cite{bardos2021}). Finally, \cite{bardos2019,gwiazda} considered the analogue of the first half of Onsager's conjecture for the conservation of entropy and other companion laws in the general class of conservation laws. In particular, it established the universality of the exponent $\frac{1}{3}$ for this case (see also \cite{bardos2019-2} and references therein).

To establish the second half of the Onsager conjecture, namely the existence of non energy-conserving H\"older continuous weak solutions to the three-dimensional Euler equations, with H\"older exponent $\theta < \frac{1}{3}$ techniques from convex integration were implemented.
In \cite{isettproof} weak solutions with regularity up to $\frac{1}{3}$ were constructed which do not conserve energy, after gradual success in a sequence of papers \cite{scheffer,shnirelman,lellisinclusion,lellisadmissibility,lelliscontinuous,lellisdissipative,isettthesis,buckmaster,buckmasteralmost,daneri}. In \cite{buckmasteradmissible} energy-dissipating solutions were constructed with regularity less than $\frac{1}{3}$, while in \cite{novack} this was achieved in Besov spaces using an intermittent convex integration scheme.

It is natural to ask whether Onsager's conjecture can be generalised to other models, as was done in \cite{bardos2019,bardos2019-2,gwiazda,beekiealpha,boutroshydrostatic} (see also references therein and \cite{hezi}). Moreover, in recent years active scalar equations (including the surface quasi-geostrophic equation) have been the focus of several works, see for example \cite{buckmastersqg,akramov,isettsqg,novackquasigeostrophic,chengsqg} (and references therein). For many of these equations the Onsager conjecture remains open in the full range of exponents, including the 2D Euler equations \cite{buckmasterreview,buckmastersqg}.

In this paper we study several subgrid scale turbulence models (also referred to as the $\alpha$ models), such as the Euler-$\alpha$ equations. This model is the inviscid version of the Navier-Stokes-$\alpha$ equations (also known as the viscous Camassa-Holm equations) \cite{foias,camassaholm}.

The subgrid scale turbulence models have been studied because they approximate the Navier-Stokes equations. In the viscous case these models where shown to be successful subgrid scale models of turbulence when tested against experimental data. In particular, they have been shown to capture the right statistics of the Navier-Stokes equations for scales larger than the regularisation parameter $\alpha$.

Moreover, the Navier-Stokes-$\alpha$ equations have been found to be a good closure model of turbulence for channel and pipe flow (if one chooses a Helmholtz smoothing kernel), as there is agreement with experimental data \cite{bardina}. Further details can be found in \cite{camassaholm,foias,chen,chen2,chen3}.

Analytically, the viscous $\alpha$-models have been shown to be globally well-posed. Moreover, when one lets the regularisation parameter $\alpha$ go to zero they formally converge to the Navier-Stokes equations. This limit can be made rigorous. In fact, one can show that for a subsequence $(\alpha_j)$ of the regularisation parameter going to zero, the unique solutions of the Navier-Stokes-$\alpha$ equations converge to a Leray-Hopf weak solution of the Navier-Stokes equations (globally in time). In addition, one can prove that the solution of the Navier-Stokes-$\alpha$ model converges to the corresponding strong solution of the Navier-Stokes equations on its interval of existence. For further details see \cite{foias,camassaholm}.

The Euler-$\alpha$ equations also appear in several contexts. For example, the equations arise as a description of the geodesics on the volume-preserving diffeomorphism group of the $H^1$ norm \cite{ratiu,holm,shkoller}. Finally, they also can be derived from a specific choice of stress tensor for inviscid second-grade fluids \cite{busuiocsecond}.

Over the years other subgrid scale $\alpha$-models have been considered. Examples include the Bardina model \cite{bardina}, the Leray-$\alpha$ model \cite{lerayalpha}, the modified Leray-$\alpha$ model \cite{modifiedleray}, the Clark-$\alpha$ model \cite{clarkalpha} and
the Navier-Stokes-Voigt model \cite{voigt}. Finally, one can consider $\alpha$-type models for the MHD (magnetohydrodynamic) equations \cite{linshiz}.

In this paper we will study the conservation of energy-like quantities for the inviscid versions of these models and in particular deduce the corresponding Onsager exponents, which are the threshold Besov or H\"older exponents for energy conservation. The conserved quantity is not always the $L^2$ norm, but can also be the $H^1$ norm instead. What is important to emphasise is that in several cases, the exponent found is smaller than $1/3$, unlike for the models considered in \cite{bardos2019,bardos2019-2} (and references therein).

First we observe that the results of \cite{bardos2019-2,bardos2019} do not apply because the nonlinearity in the $\alpha$-models is nonlocal. For these models the advective velocity is more regular compared to the incompressible Euler equations, i.e. the nonlinearity is smoother. This means that compared to the Euler equations the $\alpha$-models have lower Onsager exponents and therefore weaker regularity conditions for energy conservation are expected.

The different subgrid scale models have different regularity thresholds due to several factors, such as the type of conserved quantity, the way the nonlinearity is regularised etc. This will be discussed further in the conclusion and in Remark \ref{originremark} below.

It is also important to mention that for the Leray-$\alpha$ MHD model \cite{linshiz} it is possible to `trade regularity assumptions', i.e. there are several ways to formulate the criterion for energy conservation. To be more precise, it is possible to weaken the conditions on the magnetic field $B$ if one strengthens the conditions on the velocity field $v$ such that energy conservation is still ensured (and vice versa). In other words, there is some flexibility when the conditions for conservation of the energy-like quantity are formulated.
A similar mechanism can be observed for the ordinary MHD equations \cite{buckmasterreview}.

In this paper we eventually consider the inviscid models:
\begingroup
\allowdisplaybreaks
\begin{align}
&\partial_t v + \nabla \cdot (u \otimes v) + \nabla p = 0, \tag{Leray-$\alpha$ model} \\
&\partial_t v + \nabla \cdot (u \otimes v) + \sum_{j=1}^3 v_j \nabla u_j + \nabla p = 0, \tag{Navier-Stokes-$\alpha$ model} \\
&\partial_t v + \nabla \cdot (u \otimes u) + \nabla p = 0,  \tag{Navier-Stokes-Voigt} \\
&\partial_t v + \nabla \cdot (u \otimes u) + \nabla p = 0,  \tag{Bardina model} \\
&\partial_t v +  \nabla \cdot (v \otimes u) + \nabla p = 0,  \tag{Modified Leray-$\alpha$} \\
&\partial_t v +  \nabla \cdot (u \otimes v) + \nabla \cdot (v \otimes u) - \nabla \cdot (u \otimes u) \nonumber \\
&- \alpha^2 \nabla \cdot (\nabla u \cdot \nabla u^T ) + \nabla p = 0, \tag{Clark-$\alpha$ model} \\
&\begin{cases}
\partial_t v + (u \cdot \nabla) v  + \nabla p + \frac{1}{2} \nabla \lvert B \rvert^2 = (B \cdot \nabla) B, \\
\partial_t B + (u \cdot \nabla) B - (B \cdot \nabla )v = 0.
\end{cases}\tag{Leray-$\alpha$ MHD model}
\end{align}
\endgroup
Throughout this paper the models will be subject to periodic boundary conditions on the three-dimensional flat torus $\mathbb{T}^3$. Therefore the domains will sometimes be omitted from now on. These models are incompressible, i.e.
\begin{equation*}
\nabla \cdot u = \nabla \cdot v = 0.
\end{equation*}
In addition $u$ is a Helmholtz regularisation of $v$, which is
\begin{equation} \label{helmholtz}
v = (1 - \alpha^2 \Delta) u.
\end{equation}
We also introduce the following Sobolev norm
\begin{equation} \label{H1norm}
\lVert u \rVert_{H^1}^2 = \lVert u \rVert_{L^2}^2 + \alpha^2 \lVert \nabla u \rVert_{L^2}.
\end{equation}

For all these models we consider the inviscid case $\nu = 0$ (and the irresistive case $\eta = 0$ for the Leray-$\alpha$ MHD model). Under that assumption the Navier-Stokes-$\alpha$ equations are called the Euler-$\alpha$ equations and the Navier-Stokes-Voigt model is called the Euler-Voigt model. Note that the inviscid Bardina model and the Euler-Voigt model are the same.

The Euler-Voigt model is of interest because it is an inviscid regularisation of the Euler equations. To be more precise, it has the same form as the Euler equations but with an additional term $-\alpha^2 \Delta \partial_t u$. Note that this additional term is independent of viscosity and that the Euler-Voigt model is globally well-posed with and without viscosity. It can be shown by classical methods that the unique weak solution of the Euler-Voigt model conserves energy \cite{bardina,larios-petersen-titi-wingate}. For this reason, there is no analogue of the Onsager conjecture for this model, which is why the Euler-Voigt model will not be considered any further in this work.

Our approach in this paper will be based on \cite{duchon}, and the results were part of \cite{boutrosthesis}. We first derive an equation describing the local energy balance. This equation contains a `defect term' which captures a (potential) lack of smoothness of the solution and physically describes the energy flux. In particular, it captures the possible dissipation of energy for weak solutions when the solutions are not smooth enough.

In our framework (cf. \cite{boutrosthesis}), we will show that under sufficiently strong regularity assumptions, this term will be zero. This in turn rules out the dissipation (or non-conservation) of energy.

The procedure of deriving the Onsager exponent will be exactly the same for all these models. Therefore we will only work out the detailed proof for the inviscid Leray-$\alpha$ model and the Euler-$\alpha$ equations. These models have a different type of conserved quantity (the $L^2$ norm and the $H^1$ norm respectively) and their weak formulations are quite different in nature. In the main text we just state the results for the three other models (the Modified Leray-$\alpha$, Clark-$\alpha$ and Leray-$\alpha$ MHD models) for the sake of brevity and because the cases of the Leray-$\alpha$ model and the Euler-$\alpha$ equations are illustrative of how the proofs work for the other models. The detailed proofs for the other models are then given in the appendices.

The conserved quantity for the Leray-$\alpha$ model is the $L^2$ norm in space of $v$. The conserved quantity for the Euler-$\alpha$, Modified Leray-$\alpha$ and Clark-$\alpha$ models is the $H^1$ norm in space of $u$ given by equation \eqref{H1norm}. The conserved quantity for the inviscid and irresistive Leray-$\alpha$ MHD model is given by
\begin{equation*}
\lVert v(\cdot , t) \rVert^2_{L^2} + \lVert B (\cdot, t) \rVert^2_{L^2}.
\end{equation*}
The different types of conserved quantities do not really alter the proof significantly. However, it is important to mention that the regularity assumptions for the turbulence models differ. For the inviscid Leray-$\alpha$ model we assume that the weak solutions are in the space $v \in L^\infty ((0,T); L^2 (\mathbb{T}^3))$, while for the Euler-$\alpha$ equations, the modified Leray-$\alpha$ and Clark-$\alpha $ models the regularity assumption is $u \in L^\infty ((0,T); H^1 (\mathbb{T}^3))$. For the MHD Leray-$\alpha$ model the assumption is $v, B \in L^\infty ((0,T); L^2 (\mathbb{T}^3))$.

We now outline the rest of the paper. In section \ref{leraysection} we prove the first half of the Onsager conjecture for the Leray-$\alpha$ model, and in section \ref{eulersection} we prove the result for the Euler-$\alpha$ equations. In section \ref{resultssection} we provide an overview of the sufficient conditions for energy conservation for all the models and provide some interpretation of the results.

We conclude in section \ref{conclusion}, where we discuss what determines the actual value of the Onsager exponent for a given model. In addition, we show that for a generalisation of the inviscid Leray-$\alpha$ model that there is a linear relationship between the degree of regularisation of the nonlinearity and the Onsager exponent (i.e. the threshold for energy conservation).

In the appendices we provide the detailed proofs for the other models we consider. Namely, in section \ref{modifiedleraysection} we consider the modified Leray-$\alpha$ model, while in section \ref{clarksection} we consider the Clark-$\alpha$ model. Finally, in section \ref{lerayMHDsection} we prove the result for the Leray-$\alpha$ MHD model.

We emphasise that in this paper we focus on the first half of Onsager's conjecture. The recent paper \cite{beekiealpha} treats both the first as well as the second half for the Euler-$\alpha$ equations.

This work is dedicated to the memory of Professor Charles R. Doering, who was a
great scientist and a very dear friend.
\section{The Leray-$\alpha$ model} \label{leraysection}
In this section, we prove the first half of Onsager's conjecture for the Leray-$\alpha$ model. We start by deriving an equation of local energy balance. This equation describes the time evolution of the conserved quantity in terms of an energy flux (which will be referred to as the defect term). We will first introduce some notation. We let $\phi : \mathbb{R}^3 \rightarrow \mathbb{R}$ be a standard $C^\infty_c $ radial mollifier satisfying $\int_{\mathbb{R}^3} \phi (x) = 1$ and let $\phi_\epsilon$ be defined by
\begin{equation*}
\phi_\epsilon \coloneqq \frac{1}{\epsilon^3} \phi \bigg( \frac{x}{\epsilon} \bigg).
\end{equation*}
We also introduce the notation $v^\epsilon \coloneqq v * \phi_\epsilon$. Throughout this paper we will be using the Einstein summation convention. Finally we will use the brackets $\langle \cdot, \cdot \rangle$ to denote the action of a distribution on a test function. We will first introduce the notion of a weak solution for the inviscid Leray-$\alpha$ model.
\begin{definition} \label{weaksolutiondefinition}
A pair of functions $v,p \in L^\infty ((0,T); L^2 (\mathbb{T}^3))$ is called a weak solution of the inviscid Leray-$\alpha$ model if for all $\psi \in \mathcal{D} (\mathbb{T}^3 \times (0,T) ; \mathbb{R}^3)$ and $\chi \in \mathcal{D} (\mathbb{T}^3 \times (0,T) ; \mathbb{R})$ it holds that
\begin{align}
&\int_0^T \int_{\mathbb{T}^3} v_i \partial_t \psi_i dx dt + \int_0^T \int_{\mathbb{T}^3} v_j u_i  \partial_i \psi_j dx dt + \int_0^T \int_{\mathbb{T}^3} p \partial_i \psi_i dx dt = 0, \label{weaksolution} \\
&\int_0^T \int_{\mathbb{T}^3 } u_i \partial_i \chi dx = 0.
\end{align}
Note that the pressure is defined up to a constant. We fix this constant by assuming that
\begin{equation*}
\int_{\mathbb{T}^3} p(x,t) dx = 0.
\end{equation*}
\end{definition}
\begin{remark} \label{pressureremark}
For the Leray-$\alpha$ model we can derive the following equation for the pressure (which holds in the sense of distributions for test functions in $\mathcal{D} (\mathbb{T}^3 \times (0,T))$)
\begin{equation*}
\Delta p =- (\nabla \otimes \nabla) : (u \otimes v),
\end{equation*}
where we recall that $v \in L^\infty ((0,T); L^2 (\mathbb{T}^3))$ and $u=(I-\alpha^2 \Delta)^{-1} v \in L^\infty ((0,T); H^2 (\mathbb{T}^3) )$. By elliptic regularity it follows that
\begin{align*}
\lVert p \rVert_{L^2 }&\leq c \lVert u \otimes v \rVert_{L^2} \leq c \lVert u \rVert_{L^\infty} \lVert v \rVert_{L^2} \leq c \lVert u \rVert_{H^2} \lVert v \rVert_{L^2} \leq c \lVert v \rVert_{L^2}^2.
\end{align*}
for some constant $c$ (which may change from line to line). Note that we have used the Sobolev embedding theorem in three dimensions, $H^2(\mathbb{T}^3) \subset L^\infty(\mathbb{T}^3)$, in the above. For this reason we introduced the assumption that $p \in L^\infty ((0,T); L^2 (\mathbb{T}^3))$ in the definition of a weak solution for the Leray-$\alpha$ model.
\end{remark}
In fact, instead of $\mathcal{D} (\mathbb{T}^3 \times (0,T))$ one can consider a larger space of test functions, as we prove in the next lemma.
\begin{lemma} \label{generaltestfunctions}
A weak solution of the inviscid Leray-$\alpha$ model also satisfies identity \eqref{weaksolution} for functions $\psi \in W^{1,1}_0 ( (0,T); H^1 (\mathbb{T}^3)) $.
\end{lemma}
\begin{proof}
For every $\psi \in W^{1,1}_0 ( (0,T); H^1 (\mathbb{T}^3)) $ we know that there exists a sequence of test functions $\psi_m \in \mathcal{D} ( \mathbb{T}^3 \times (0,T))$ that converge to $\psi$ in $W^{1,1}_0 ((0,T); H^1(\mathbb{T}^3))$. For each of the functions $\psi_m$, identity \eqref{weaksolution} holds. We observe that $v \cdot \partial_t \psi_m $ converges to $v \partial_t \psi$ in $L^1 ((0,T); L^{3/2} (\mathbb{T}^3))$ as $m \rightarrow \infty$. In particular, this implies that

\begin{equation*}
 \int_0^T \int_{\mathbb{T}^3} v \cdot \partial_t \psi_m dx dt \xrightarrow[]{m \rightarrow \infty}  \int_0^T \int_{\mathbb{T}^3} v \cdot \partial_t \psi dx dt.
\end{equation*}
Similarly, one can see that $v_j u_i \partial_i (\psi_m)_j$ converges to $v_j u_i \partial_i \psi_j$ in $L^1 ((0,T) \times \mathbb{T}^3)$ as $m \rightarrow \infty$ (by using that $H^2 (\mathbb{T}^3) \subset L^\infty (\mathbb{T}^3)$). Therefore, we have
\begin{equation*}
\int_0^T \int_{\mathbb{T}^3} v_j u_i  \partial_i (\psi_{m})_j dx dt \xrightarrow[]{m \rightarrow \infty} \int_0^T \int_{\mathbb{T}^3} v_j u_i  \partial_i \psi_j dx dt.
\end{equation*}
Finally, as mentioned in Remark \ref{pressureremark}, we know that $p \in L^\infty ( (0,T); L^2 (\mathbb{T}^3))$. Therefore $p \partial_i (\psi_m)$ converges to $p \partial_i \psi$ in $L^1 (\mathbb{T}^3 \times (0,T))$ as $m \rightarrow \infty$. Thus
\begin{equation*}
\int_0^T \int_{\mathbb{T}^3} p \partial_i (\psi_m) dx dt \xrightarrow[]{m \rightarrow \infty} \int_0^T \int_{\mathbb{T}^3} p \partial_i \psi_i dx dt.
\end{equation*}
As a result identity \eqref{weaksolution} also holds for test functions $\psi \in W^{1,1}_0 ((0,T); H^1 (\mathbb{T}^3))$.
\end{proof}
We now turn to establishing the equation of local energy balance.
\begin{theorem}[Equation of energy balance] \label{energyequationtheorem}
Let $v \in L^\infty ( (0,T); L^2 (\mathbb{T}^3))$ be a weak solution of the inviscid Leray-$\alpha$ model (as introduced in Definition \ref{weaksolutiondefinition}). Then $v$ satisfies the following equation of energy (which holds in the sense of distributions for test functions in $\mathcal{D} ( \mathbb{T}^3 \times (0,T))$)
\begin{equation} \label{lerayalpha}
\partial_t (\lvert v \rvert^2) + 2 \nabla \cdot ( p  v) + \nabla \cdot ( \lvert v \rvert^2 u ) + D_1(u,v) = 0,
\end{equation}
where the defect term $D_1 (u,v)$ is
\begin{equation*}
D_1 (u,v) (x,t) = \frac{1}{2} \lim_{\epsilon \rightarrow 0} \int_{\mathbb{R}^3} \nabla \phi_\epsilon (\xi) \cdot \delta u (\xi; x,t) \lvert  \delta v  (\xi; x,t) \rvert^2 d \xi,
\end{equation*}
and the notation means
\begin{equation*}
\delta w (\xi; x,t) \coloneqq w (x + \xi,t) - w(x,t).
\end{equation*}
Note that the defect term is independent of the choice of mollifer.

\end{theorem}

\begin{proof}
Mollifying the equation of the Leray-$\alpha$ model by convolution with $\phi_\epsilon$ gives that
\begin{equation} \label{mollifiedleray}
\partial_t v^\epsilon + \partial_i (u_i v)^\epsilon + \nabla p^\epsilon = 0.
\end{equation}
This equation holds pointwise a.e. in $\mathbb{T}^3 \times (0,T)$. This is because $v^\epsilon$ is $C^\infty$ in space and we can use the equation to gain time regularity for $v^\epsilon$. In particular, we observe that $\partial_i (u_i v_j)^\epsilon + \nabla p^\epsilon$ is in $L^\infty ( (0,T); C^\infty (\mathbb{T}^3) )$ which implies that $v^\epsilon \in W^{1, \infty} ( (0,T); C^\infty (\mathbb{T}^3))$. We take a test function $\chi \in \mathcal{D} ( \mathbb{T}^3 \times (0,T) ; \mathbb{R})$ and take the dot product of equation \eqref{mollifiedleray} with $v \chi \in L^\infty ((0,T); L^2 (\mathbb{T}^3))$. The resulting equation is
\begin{equation} \label{mollifiedleraymultiplied}
\chi v \cdot \partial_t v^\epsilon + \chi v_j \partial_i (u_i v_j)^\epsilon + \chi v \cdot \nabla p^\epsilon = 0,
\end{equation}
which holds almost everywhere in $\mathbb{T}^3 \times (0,T)$. Now we let the test function $\chi v^\epsilon$ act on the original inviscid Leray-$\alpha$ equation. Note that we can use $\chi v^\epsilon$ as a test function since it lies in $W^{1,\infty}_0 ((0,T); C^\infty (\mathbb{T}^3)) \subset W^{1,1}_0 ((0,T); H^1 (\mathbb{T}^3))$ (so we can apply Lemma \ref{generaltestfunctions}). We subtract the resulting equation from the integrated version of equation \eqref{mollifiedleraymultiplied}, which gives (where we have used the weak formulation from Definition \ref{weaksolutiondefinition})
\begin{align*}
\int_0^T \int_{\mathbb{T}^3} \bigg[ \chi v \cdot \partial_t v^\epsilon - v \partial_t (\chi v^\epsilon) + \chi v_j \partial_i (u_i v_j)^\epsilon - v_j u_i \partial_i (\chi v_j^\epsilon) + \chi v \cdot \nabla p^\epsilon - p \partial_i (v^\epsilon_i \chi) \bigg] dx dt = 0.
\end{align*}
Now we rewrite the different parts of the above equation, we see that
\begin{equation*}
\int_0^T \int_{\mathbb{T}^3} \bigg[ \chi v \cdot \partial_t v^\epsilon - v \cdot \partial_t (\chi v^\epsilon) \bigg] dx dt = - \int_0^T \int_{\mathbb{T}^3} (v \cdot v^\epsilon) \partial_t \chi dx dt = \langle \partial_t (v \cdot v^\epsilon), \chi \rangle.
\end{equation*}
To handle the advective terms, we introduce a defect term
\begin{equation*}
D_{1,\epsilon} (u,v) (x,t) \coloneqq \frac{1}{2} \int_{\mathbb{R}^3} \nabla_\xi \phi_\epsilon (\xi) \cdot \delta u (\xi ;x,t) \lvert \delta v (\xi; x,t ) \rvert^2 d \xi,
\end{equation*}
observe that
\begin{equation*}
D_{1,\epsilon} (u,v) = - \frac{1}{2} \partial_i (u_i v_j v_j)^\epsilon + \frac{1}{2} u_i \partial_i ( v_j v_j)^\epsilon + v_j \partial_i (v_j u_i)^\epsilon - u_i v_j \partial_i v_j^\epsilon,
\end{equation*}
which holds almost everywhere in $\mathbb{T}^3 \times (0,T)$. Now we can rewrite the advective terms as
\begingroup
\allowdisplaybreaks
\begin{align*}
&\int_0^T \int_{\mathbb{T}^3} \bigg[\chi v_j \partial_i (u_i v_j)^\epsilon - v_j u_i \partial_i (\chi v_j^\epsilon) \bigg] dx dt = \int_0^T \int_{\mathbb{T}^3} \bigg[\chi v_j \partial_i (u_i v_j)^\epsilon - \chi v_j u_i \partial_i  v_j^\epsilon - v_j u_i v_j^\epsilon \partial_i \chi  \bigg] dx dt\\
&= \int_0^T \int_{\mathbb{T}^3} \bigg[\chi D_{1,\epsilon} (u,v) + \frac{1}{2} \chi \partial_i (u_i v_j v_j)^\epsilon - \frac{1}{2} \chi u_i \partial_i  ( v_j v_j)^\epsilon - v_j u_i v_j^\epsilon \partial_i \chi  \bigg] dx dt \\
&= \int_0^T \int_{\mathbb{T}^3} \bigg[\chi D_{1,\epsilon} (u,v) + \frac{1}{2}\bigg(( v_j v_j)^\epsilon u_i  -(u_i v_j v_j)^\epsilon \bigg) \partial_i \chi   - v_j u_i v_j^\epsilon \partial_i \chi  \bigg] dx dt \\
&= \bigg\langle D_{1,\epsilon} (u,v) + \nabla \cdot ( (\lvert v \rvert^2 u)^\epsilon - (\lvert v \rvert^2)^\epsilon u ) + \nabla \cdot ( (v \cdot v^\epsilon) u) , \chi \bigg\rangle.
\end{align*}
\endgroup
Next, we show that $( v_j v_j)^\epsilon u_i  -(u_i v_j v_j)^\epsilon \rightarrow 0$ in $L^\infty ((0,T); L^1 (\mathbb{T}^3))$ as $\epsilon \rightarrow 0$. One can see that since $u_i (\cdot, t) \in H^2 (\mathbb{T}^3) \subset L^\infty (\mathbb{T}^3)$ and $v_j (\cdot, t) \in L^2 (\mathbb{T}^3)$ then $u_i v_j v_j (\cdot, t) \in L^1 (\mathbb{T}^3)$, thus
\begin{align*}
\lVert (u_i v_j v_j)^\epsilon - u_i ( v_j v_j)^\epsilon  \rVert_{L^1 } &\leq \lVert (u_i v_j v_j)^\epsilon - u_i v_j v_j  \rVert_{L^1 } + \lVert u_i v_j v_j - u_i ( v_j v_j)^\epsilon  \rVert_{L^1 } \\
&\leq \lVert (u_i v_j v_j)^\epsilon - u_i v_j v_j  \rVert_{L^1 } + \lVert u_i \rVert_{L^\infty } \lVert v_j v_j - ( v_j v_j)^\epsilon  \rVert_{L^{1} } \xrightarrow[]{\epsilon \rightarrow 0} 0.
\end{align*}
Consequently, the advective terms converge to $\langle D_1 (u,v) + \nabla \cdot ( \lvert v \rvert^2 u), \chi \rangle$ in the sense of distributions as $\epsilon \rightarrow 0$, where $D_1 (u,v) \coloneqq \lim_{\epsilon \rightarrow 0} D_{1,\epsilon} (u,v)$.
Finally, we get that
\begin{align*}
\int_0^T \int_{\mathbb{T}^3} \bigg[ \chi v \cdot \nabla p^\epsilon - p \partial_i (v^\epsilon_i \chi) \bigg] dx dt &= \int_0^T \int_{\mathbb{T}^3} \bigg[ \psi v \cdot \nabla p^\epsilon - p v^\epsilon_i \partial_i  \chi \bigg] dx dt \\
&= - \int_0^T \int_{\mathbb{T}^3} \bigg[ p^\epsilon v_i \partial_i \chi + p v^\epsilon_i \partial_i  \chi \bigg] dx dt \\
&= \langle \nabla \cdot (p^\epsilon v + p v^\epsilon), \chi \rangle.
\end{align*}
Combining these results we get the following equation of energy (which holds in the sense of distributions with test functions in $\mathcal{D} ( \mathbb{T}^3 \times (0,T))$)
\begin{equation*}
\partial_t (v \cdot v^\epsilon) + D_{1,\epsilon} (u,v) + \nabla \cdot ( (v \cdot v^\epsilon) u) - \frac{1}{2} \nabla \cdot ((\lvert v \rvert^2)^\epsilon u) + \frac{1}{2} \nabla \cdot (\lvert v \rvert^2 u)^\epsilon + \nabla \cdot (p^\epsilon v + p v^\epsilon) = 0.
\end{equation*}
From this equation we can obtain that the limit $\lim_{\epsilon \rightarrow 0} D_{1,\epsilon} (u)$ is a well-defined distributional limit, as all the other terms converge. As was argued before, we will have that $-\frac{1}{2} \nabla \cdot ((\lvert v \rvert^2)^\epsilon u) + \frac{1}{2} \nabla \cdot (\lvert v \rvert^2 u)^\epsilon \xrightarrow[]{\epsilon \rightarrow 0} 0$ as a distributional limit. Moreover, since $v \cdot v^\epsilon \rightarrow \lvert v \rvert^2$ in $L^\infty ((0,T); L^1 (\mathbb{T}^3))$, it follows that $\partial_t (v \cdot v^\epsilon) \rightarrow \partial_t \lvert v \rvert^2$ in the distributional limit.
We observe that we can write the following equation for the defect term
\begin{equation*}
D_1 (u,v)  = - \partial_t \lvert v \rvert^2 - \nabla \cdot ( \lvert v \rvert^2 u)  - \nabla \cdot (2 p v).
\end{equation*}
Observe that the right-hand side is completely independent of the choice of mollifier, as a result the defect term is independent of the choice of mollifier.
We end up with the equation
\begin{equation*}
\partial_t \lvert v \rvert^2 + D_1 (u,v) + \nabla \cdot ( \lvert v \rvert^2 u)  + \nabla \cdot (2 p v) = 0,
\end{equation*}
which concludes the proof.
\end{proof}
\begin{proposition}  \label{zerodefect}
Let $v$ be a weak solution of the inviscid Leray-$\alpha$ model. Moreover, we assume that
\begin{equation} \label{leraycondition}
\int_{\mathbb{T}^3} \lvert \delta u \rvert \lvert \delta v \rvert^2 dx \leq C(t) \lvert \xi \rvert \sigma_1 ( \lvert \xi \rvert),
\end{equation}
where $C \in L^1 (0,T)$ and $\sigma_1 \in L^\infty_{\mathrm{loc}} (\mathbb{R})$, with the property that $\sigma_1 (\lvert \xi \rvert) \rightarrow 0$ as $\lvert \xi \rvert \rightarrow 0$. Then $D_1 (u,v) = 0$.
\end{proposition}
\begin{proof}
By the change of variable $\xi = \epsilon z$ one can calculate that
\begingroup
\allowdisplaybreaks
\begin{align*}
\int_0^T \int_{\mathbb{T}^3}& \lvert D_{1,\epsilon } (u,v) (x,t) \rvert dx \leq \int_0^T \int_{\mathbb{R}^3} \lvert \nabla_\xi \phi_\epsilon (\xi ) \rvert \int_{\mathbb{T}^3} \lvert \delta u \rvert \lvert \delta v \rvert^2 dx d\xi \\ & \leq \int_0^T C(t) d t\int_{\mathbb{R}^3} \lvert \nabla_\xi \phi_\epsilon (\xi ) \rvert \lvert \xi \rvert \sigma_1 ( \lvert \xi \rvert) d\xi
= \int_0^T C(t) d t\int_{\mathbb{R}^3} \lvert \nabla_z \phi (z ) \rvert \lvert z \rvert \sigma_1 ( \epsilon \lvert z \rvert) dz \xrightarrow[]{\epsilon \rightarrow 0} 0,
\end{align*}
\endgroup
where we used the Lebesgue dominated convergence theorem to justify the limit in the last step as $\epsilon \rightarrow 0$. Moreover, note that $\lvert \delta u \rvert \lvert \delta v \rvert^2 \in L^\infty ((0,T); L^1 (\mathbb{T}^3))$ so one can swap the order of integration in the first inequality. Consequently, one has  $D_1 (u,v)=0$.
\end{proof}
In order to prove the sufficient condition for energy conservation, we will use the following inequality for functions in Besov spaces (see, e.g., \cite{constantin,leoni})
\begin{equation} \label{besovinequality}
\lVert f (\cdot + \xi) - f(\cdot) \rVert_{L^p} \leq C \lvert \xi \rvert^\theta \lVert f \rVert_{B^\theta_{p,\infty}},
\end{equation}
which holds for $1 \leq p \leq \infty$, $\theta > 0$ and some constant $C$.
\begin{proposition} \label{suffconditionleray}
If $v$ is a weak solution of the Leray-$\alpha$ model such that $v \in L^3 ((0,T); B^{\theta}_{3,\infty} (\mathbb{T}^3))$ with $\theta > 0$, then it holds that $D_1 (u,v) = 0$. In particular, the solution conserves energy.
\end{proposition}
\begin{proof}
Condition \eqref{leraycondition} can be verified quite easily to see that
\begin{equation*}
\int_{\mathbb{T}^3} \lvert \delta u \rvert \lvert \delta v \rvert^2 dx \leq \lvert \xi \rvert^{1 + 2 \theta} \lVert v \rVert_{B^\theta_{3,\infty}}^3,
\end{equation*}
which follows from using inequality \eqref{besovinequality}.
Then we can take $\sigma_1 (\lvert \xi \rvert) \coloneqq \lvert \xi \rvert^{2\theta}$, which tends to zero as $\lvert \xi \rvert \to 0$. Therefore, the proof can be completed thanks to  Proposition \ref{zerodefect}.
\end{proof}

Having analysed the defect term, we can prove the conservation of energy. This comes down to a suitable choice of test function (which ensures that the divergence term vanishes) and then using the Lebesgue differentiation theorem to prove conservation of the energy.
\begin{theorem}[Conservation of energy] \label{conservationtheorem}
Suppose $v \in L^\infty ( (0,T); L^2 (\mathbb{T}^3))$ is a weak solution of the inviscid Leray-$\alpha$ model and in addition it satisfies the conditions of Proposition \ref{zerodefect} (for example when $v \in L^3 ((0,T); B^s_{3,\infty} (\mathbb{T}^3))$ with $s > 1$, by Proposition \ref{suffconditionleray}). Then conservation of energy holds, i.e.
\begin{equation*}
\lVert v(t_1, \cdot) \rVert_{L^2} = \lVert v(t_2, \cdot) \rVert_{L^2},
\end{equation*}
for almost all $t_1, t_2 \in (0,T)$.
\end{theorem}
\begin{proof}
Using Proposition \ref{zerodefect} together with the local equation of energy established in Theorem \ref{energyequationtheorem} gives us that
\begin{equation*}
\partial_t (\lvert v \rvert^2) + 2 \nabla \cdot ( p  v) + \nabla \cdot ( \lvert v \rvert^2 u )  = 0.
\end{equation*}
This equation is valid in the distributional sense, i.e. for any test function $\psi \in \mathcal{D} ( (0,T) \times \mathbb{T}^3)$ it holds that
\begin{equation} \label{distributionalequation}
\int_0^T \int_{\mathbb{T}^3} \frac{1}{2} \lvert v \rvert^2 \partial_t \psi dx dt= - \int_0^T \int_{\mathbb{T}^3} \nabla \psi \cdot \bigg(  \frac{1}{2} \lvert v \rvert^2 u + v p \bigg) \bigg) dx dt.
\end{equation}
We want to show that the $L^2$ norm of $v$ at time $t_1$ equals the $L^2$ norm at $t_2$ (where $t_1 < t_2$). Let $\varphi: \mathbb{R} \rightarrow \mathbb{R}$ be a standard $C^\infty_c$ mollifier with $\int_{\mathbb{R}} \varphi (t) dt = 1$ with support contained in $[-1,1]$. We introduce the notation
\begin{equation*}
\varphi_\epsilon (t) \coloneqq \frac{1}{\epsilon} \varphi \bigg( \frac{t}{\epsilon} \bigg).
\end{equation*}
We make the following choice of test function
\begin{equation*}
\psi_1 (t) = \int_0^t \varphi_{\epsilon} (t'-t_1) - \varphi_{\epsilon} (t' - t_2) dt',
\end{equation*}
for $\epsilon$ sufficiently small. Because of the properties of mollifiers, one can see that the function $\psi_1(t)=0$ for $t\in (0, t_1-\epsilon) \cup (t_2+\epsilon, T)$, and $\psi_1(t)=1$ for $t\in (t_1+\epsilon, t_2-\epsilon)$. This means in particular that the function has compact support in $(0,T)$. Therefore equation \eqref{distributionalequation} becomes
\begin{equation*}
\int_{t_1 - \epsilon}^{t_1 + \epsilon} \int_{\mathbb{T}^3} \lvert v \rvert^2 \varphi_\epsilon (t - t_1) dx dt = \int_{t_2 - \epsilon}^{t_2 + \epsilon} \int_{\mathbb{T}^3} \lvert v \rvert^2 \varphi_\epsilon (t - t_2) dx dt.
\end{equation*}
Then we apply the Lebesgue differentiation theorem (cf., \cite{folland,zygmund}), by taking the limit $\epsilon \rightarrow 0$, which yields
\begin{equation*}
\int_{\mathbb{T}^3} \lvert v  (x,t_1) \rvert^2 dx = \int_{\mathbb{T}^3} \lvert v  (x,t_2) \rvert^2 dx
\end{equation*}
for almost every $t_1,t_2 \in (0,T)$.
Therefore we conclude that the $L^2$ norm $\lVert v(t, \cdot) \rVert_{L^2} $ is conserved for almost all $t \in (0,T)$. Thus we proved conservation of energy for weak solutions (i.e. the analogue of the first half of Onsager's conjecture for this model) under the assumptions of Proposition \ref{zerodefect}.
\end{proof}
\section{The Euler-$\alpha$ equations} \label{eulersection}
We now prove a similar result for the Euler-$\alpha$ equations, we recall that the equations are given by
\begin{align*}
&\partial_t v - \nu \Delta v + \nabla \cdot (u \otimes v) + \sum_{j=1}^3 v_j \nabla u_j + \nabla p = 0, \\
&v = u - \alpha^2 \Delta u, \quad \nabla \cdot u = \nabla \cdot v = 0.
\end{align*}
The definition of a weak solution is slightly different compared to the Leray-$\alpha$ model, since the regularity assumption is now $u \in L^\infty ((0,T) ; H^1 (\mathbb{T}^3))$ instead of $v \in L^\infty ((0,T); L^2 (\mathbb{T}^3))$.
\begin{definition} \label{euleralphaweaksolution}
A pair of functions $u \in L^\infty ((0,T); H^1 (\mathbb{T}^3))$ and $p \in L^\infty ((0,T); L^1 (\mathbb{T}^3))$ is called a weak solution of the Euler-$\alpha$ equations if it satisfies the equations (for all $\psi \in \mathcal{D} (\mathbb{T}^3 \times (0,T) ; \mathbb{R}^3)$ and $\chi \in \mathcal{D} (\mathbb{T}^3 \times (0,T) ; \mathbb{R})$)
\begingroup
\allowdisplaybreaks
\begin{align*}
&\int_0^T \int_{\mathbb{T}^3 } u \cdot \partial_t \psi dx dt + \alpha^2 \int_0^T \int_{\mathbb{T}^3 } \nabla u : \nabla \partial_t \psi dx dt + \int_0^T \int_{\mathbb{T}^3 } (u \otimes u) : \nabla \psi dx dt \\
&+ \alpha^2 \int_0^T \int_{\mathbb{T}^3 } (\partial_i u \otimes \partial_i u) : \nabla \psi dx dt + \alpha^2 \int_0^T \int_{\mathbb{T}^3 } (u \otimes \partial_i u) : \nabla \partial_i \psi dx dt  \\
&- \int_0^T \int_{\mathbb{T}^3 } u_j \partial_i u_j \psi_i dx dt - \alpha^2 \int_0^T \int_{\mathbb{T}^3 } \partial_k u_j \partial_i u_j \partial_k \psi_i dx dt + \frac{1}{2} \alpha^2 \int_0^T \int_{\mathbb{T}^3 } \partial_k u_j \partial_k u_j \partial_i \psi dx dt  \\
&+ \int_0^T \int_{\mathbb{T}^3 } p \nabla \cdot \psi dx dt = 0,  \\
&\int_0^T \int_{\mathbb{T}^3 } u_i \partial_i \chi dx dt = 0.
\end{align*}
\endgroup
Once again, the pressure will be determined uniquely by assuming that
\begin{equation*}
\int_{\mathbb{T}^3} p(x,t) dx = 0.
\end{equation*}
\end{definition}
As before, it is possible to extend the space of test functions for these models, like was done in Lemma \ref{generaltestfunctions}.
\begin{remark}
We observe that for the Euler-$\alpha$ equations it is possible to write down the following distributional equation for the pressure (where $\psi \in \mathcal{D} (\mathbb{T}^3 \times (0,T); \mathbb{R})$)
\begin{align*}
&\langle  \Delta p , \psi \rangle + \langle (\nabla \otimes \nabla) : (u \otimes u) , \psi \rangle  + \alpha^2 \langle (\nabla \otimes \nabla) : (\partial_i u \otimes \partial_i u) , \psi \rangle - \alpha^2 \langle \partial_i (\nabla \otimes \nabla) : (u \otimes \partial_i u) , \psi \rangle \\
&+ \frac{1}{2} \langle \Delta (\lvert u \rvert^2), \psi \rangle - \alpha^2 \langle \partial_k \partial_i (\partial_k u_j \partial_i u_j), \psi \rangle + \frac{1}{2} \alpha^2 \langle \Delta (\partial_k u_j \partial_k u_j), \psi \rangle = 0.
\end{align*}
We can conclude from this equation that if $u \in L^3 ((0,T); W^{1,3} (\mathbb{T}^3))$ then by elliptic regularity theory it follows that $p \in L^{3/2} ((0,T); L^{3/2} (\mathbb{T}^3))$, this observation will be used later.

\end{remark}

As we did with the Leray-$\alpha$ model, we now enlarge the space of test functions in the following lemma.
\begin{lemma} \label{euleralphagentest}
The weak formulation of the Euler-$\alpha$ equations still holds if $\psi \in W^{1,1}_0 ((0,T); \linebreak H^1 (\mathbb{T}^3)) \cap L^1 ((0,T); H^3 (\mathbb{T}^3))$.
\end{lemma}
\begin{proof}
The proof goes the same way as the proof of Lemma \ref{generaltestfunctions}, by using a density argument.
\end{proof}
Now we will establish the equation of local energy balance.
\begin{theorem}[Equation of energy balance]
Let $u \in L^\infty ( (0,T); H^1 (\mathbb{T}^3))$ be a weak solution of the Euler-$\alpha$ equations such that $u \in L^3 ((0,T); W^{1,3} (\mathbb{T}^3))$. Then the following equation of local energy balance holds  in the sense of distributions,
\begin{align*}
& \partial_t ( \lvert u \rvert^2) + \alpha^2 \partial_t ( \lvert \nabla u \rvert^2) - 2 \alpha^2 \partial_t \partial_i \bigg( u_j \partial_i u_j  \bigg) + 2 \alpha^2\partial_i \bigg( \partial_t u_j \partial_i u_j \bigg) + D_{2} (u)  + \nabla \cdot ( \lvert u \rvert^2 u)  \\
&+ 2\alpha^2 \nabla \cdot ( \partial_k u_j u_j \partial_k u) + \frac{1}{2} \alpha^2 D_{3} (u) + 2 \alpha^2 \nabla \cdot (\partial_k u_j \partial_k u_j u) + 2 \alpha^2 \partial_i \partial_k (u_i \partial_k u_j u_j) = 0,
\end{align*}
 where the defect terms are given by
\begingroup
\allowdisplaybreaks
\begin{align*}
D_{2} (u) (x,t) &\coloneqq \frac{1}{2} \lim_{\epsilon \rightarrow 0} \int_{\mathbb{R}^3} \nabla_\xi \phi_\epsilon (\xi) \cdot \delta u(\xi;x,t) \lvert \delta u (\xi ;x,t) \lvert^2 d \xi,\\
D_{3} (u) (x,t) &\coloneqq \lim_{\epsilon \rightarrow 0} \int_{\mathbb{R}^3} \partial_i \phi_\epsilon (\xi) \delta u_i (\xi ; x,t) \delta \partial_k u_j (\xi ;x,t) \delta \partial_k u_j (\xi ;x,t)  d \xi,
\end{align*}
\endgroup
where the limits above also hold in the sense of distribution. Moreover, the defect terms, $D_2(u), D_3(u)$, are independent of the choice of mollifier.
\end{theorem}
\begin{proof}
By mollifying the Euler-$\alpha$ equations we find that $\partial_t v^\epsilon, \partial_t u^\epsilon \in L^\infty ((0,T); C^\infty (\mathbb{T}^3))$. As a result, we can conclude that $u^\epsilon \in W^{1,\infty} ((0,T); C^\infty (\mathbb{T}^3))$. As a result, we are able to apply Lemma \ref{euleralphagentest} and take $u^\epsilon \chi$ as our test function. Mollifying our equation and multiplying it by $u \chi$ and subtracting it from the weak formulation gives us that
\begingroup
\allowdisplaybreaks
\begin{align*}
&\int_0^T \int_{\mathbb{T}^3 } u \partial_t (u^\epsilon \chi) dx dt - \int_0^T \int_{\mathbb{T}^3} u \chi \partial_t (u^\epsilon) dx dt + \alpha^2 \int_0^T \int_{\mathbb{T}^3 } \nabla u : \nabla \partial_t (u^\epsilon \chi) dx dt  \\
&- \alpha^2 \int_0^T \int_{\mathbb{T}^3} \nabla (u \chi) :  \partial_t \nabla u^\epsilon dx dt +\int_0^T \int_{\mathbb{T}^3 } (u \otimes u) : \nabla (u^\epsilon \chi) dx dt - \int_0^T \int_{\mathbb{T}^3 }  \chi u \cdot  (\nabla \cdot (u \otimes u)^\epsilon) dx dt  \\
&+ \alpha^2 \int_0^T \int_{\mathbb{T}^3} (\partial_k u \otimes \partial_k u) : \nabla (u^\epsilon \chi) dx dt - \alpha^2 \int_0^T \int_{\mathbb{T}^3} u \chi \cdot (\nabla \cdot (\partial_k u \otimes \partial_k u)^\epsilon) dx dt \\
&+ \alpha^2 \int_0^T \int_{\mathbb{T}^3 } (u \otimes \partial_k u) : \nabla \partial_k (u^\epsilon \chi) dx dt - \alpha^2 \int_0^T \int_{\mathbb{T}^3} \partial_k (u \chi) \cdot (\nabla \cdot (u \otimes \partial_k u)^\epsilon) dx dt \\
&- \int_0^T \int_{\mathbb{T}^3 } u_j \partial_i u_j (u^\epsilon_i \chi) dx dt - \int_0^T \int_{\mathbb{T}^3} u_i \chi (u_j \partial_i u_j)^\epsilon dx dt \\
&- \alpha^2 \int_0^T \int_{\mathbb{T}^3 } \partial_k u_j \partial_i u_j \partial_k (u^\epsilon_i \chi) dx dt + \alpha^2 \int_0^T \int_{\mathbb{T}^3} u_i \chi \partial_k (\partial_k u_j \partial_i u_j)^\epsilon dx dt \\
&+ \frac{1}{2} \alpha^2 \int_0^T \int_{\mathbb{T}^3 } \partial_k u_j \partial_k u_j \partial_i (u^\epsilon_i \chi) dx dt + \frac{1}{2} \alpha^2 \int_0^T \int_{\mathbb{T}^3} u_i (\partial_k u_j \partial_k u_j)^\epsilon \partial_i \chi dx dt  \\
&+ \int_0^T \int_{\mathbb{T}^3} p \nabla \cdot (u^\epsilon \chi) dx dt - \int_0^T \int_{\mathbb{T}^3} u \chi \cdot \nabla p^\epsilon dx dt = 0.
\end{align*}
\endgroup
We first deal with the time derivative terms, we get that
\begingroup
\allowdisplaybreaks
\begin{align*}
&\int_0^T \int_{\mathbb{T}^3 } u \partial_t (u^\epsilon \chi) dx dt - \int_0^T \int_{\mathbb{T}^3} u \chi \partial_t (u^\epsilon) dx dt + \alpha^2 \int_0^T \int_{\mathbb{T}^3 } \nabla u : \nabla \partial_t (u^\epsilon \chi) dx dt  \\
&- \alpha^2 \int_0^T \int_{\mathbb{T}^3} \nabla (u \chi) :  \partial_t \nabla u^\epsilon dx dt = \int_0^T \int_{\mathbb{T}^3} \bigg[ (u \cdot u^\epsilon) \partial_t \chi + \alpha^2 \nabla u : \nabla u^\epsilon \partial_t \chi \bigg] dx dt \\
&+ \alpha^2 \int_0^T \int_{\mathbb{T}^3} \bigg[ \nabla u : \nabla \chi \otimes \partial_t u^\epsilon + \nabla u : \partial_t \nabla \chi \otimes u^\epsilon - \nabla \chi \otimes u : \partial_t \nabla u^\epsilon \bigg] dx dt \\
&= \int_0^T \int_{\mathbb{T}^3} \bigg[ (u \cdot u^\epsilon) \partial_t \chi + \alpha^2 \nabla u : \nabla u^\epsilon \partial_t \chi \bigg] dx dt \\
&+ \alpha^2 \int_0^T \int_{\mathbb{T}^3} \bigg[ \nabla u : \nabla \chi \otimes \partial_t u^\epsilon + \nabla u : \partial_t \nabla \chi \otimes u^\epsilon + \partial_t \nabla \chi \otimes u : \nabla u^\epsilon + \nabla \chi \otimes \partial_t  u :\nabla u^\epsilon \bigg] dx dt
\end{align*}
\endgroup
This can be rewritten as
\begin{equation*}
- \bigg\langle \partial_t ( u \cdot u^\epsilon) + \alpha^2 \partial_t (\nabla u : \nabla u^\epsilon) - \alpha^2 \partial_t \partial_i \bigg( u_j^\epsilon \partial_i u_j + u_j \partial_i u_j^\epsilon \bigg) + \alpha^2\partial_i \bigg( \partial_t u_j^\epsilon \partial_i u_j + \partial_t u_j \partial_i u_j^\epsilon \bigg) , \chi \bigg\rangle.
\end{equation*}

Later we will choose $\chi$ such that it only depends on time, for which we will have
\[
\bigg\langle -\alpha^2 \partial_t \partial_i \bigg( u_j^\epsilon \partial_i u_j + u_j \partial_i u_j^\epsilon \bigg) + \alpha^2\partial_i \bigg( \partial_t u_j^\epsilon \partial_i u_j + \partial_t u_j \partial_i u_j^\epsilon \bigg), \chi \bigg\rangle =0.
\]

We introduce a defect term given by
\begin{equation*}
D_{2,\epsilon} (u) (x,t) \coloneqq \frac{1}{2} \int_{\mathbb{R}^3} \nabla_\xi \phi_\epsilon (\xi) \cdot \delta u(\xi;x,t) \lvert \delta u (\xi ;x,t) \lvert^2 d \xi,
\end{equation*}
then by similar reasoning to the case of the Leray-$\alpha$ model we find that
\begin{align*}
&\int_0^T \int_{\mathbb{T}^3 } (u \otimes u) : \nabla (u^\epsilon \chi) dx dt - \int_0^T \int_{\mathbb{T}^3 }  \chi u \cdot  (\nabla \cdot (u \otimes u)^\epsilon) dx dt \\
&= \int_0^T \int_{\mathbb{T}^3} \bigg[-\chi D_{2,\epsilon} (u) + \frac{1}{2}\bigg(( u_j u_j)^\epsilon u_i  -(u_i u_j u_j)^\epsilon \bigg) \partial_i \chi   + u_j u_i u_j^\epsilon \partial_i \chi  \bigg] dx dt \\
&= -\bigg\langle D_{2,\epsilon} (u) + \frac{1}{2} \nabla \cdot ( (\lvert u \rvert^2 u)^\epsilon - (\lvert u \rvert^2)^\epsilon u ) + \nabla \cdot ( (u \cdot u^\epsilon) u) , \chi \bigg\rangle.
\end{align*}
We have the terms
\begingroup
\allowdisplaybreaks
\begin{align*}
&\alpha^2 \int_0^T \int_{\mathbb{T}^3} \partial_k u_i  \partial_k u_j  \partial_i (u^\epsilon_j \chi) dx dt - \alpha^2 \int_0^T \int_{\mathbb{T}^3} u_j \chi \partial_i (\partial_k u_j  \partial_k u_i)^\epsilon dx dt \\
&+ \alpha^2 \int_0^T \int_{\mathbb{T}^3 } u_i  \partial_k u_j \partial_i \partial_k (u^\epsilon_j \chi) dx dt - \alpha^2 \int_0^T \int_{\mathbb{T}^3} \partial_k (u_j \chi) \partial_i (u_i  \partial_k u_j)^\epsilon dx dt \\
&- \alpha^2 \int_0^T \int_{\mathbb{T}^3 } \partial_k u_j \partial_i u_j \partial_k (u^\epsilon_i \chi) dx dt + \alpha^2 \int_0^T \int_{\mathbb{T}^3}  u_i \chi \partial_k (\partial_k u_j \partial_i u_j)^\epsilon dx dt \\
&= \alpha^2 \int_0^T \int_{\mathbb{T}^3 } \chi \bigg[ \partial_k u_i \partial_k u_j \partial_i u_j^\epsilon - \partial_k u_j \partial_i u_j \partial_k u_i^\epsilon \bigg] dx dt + \alpha^2 \int_0^T \int_{\mathbb{T}^3} \bigg[ \partial_k u_i \partial_k u_j u_j^\epsilon \partial_i \chi \\
&- \partial_k u_j \partial_i u_j u_i^\epsilon \partial_k \chi \bigg] dx dt + \alpha^2 \int_0^T \int_{\mathbb{T}^3} \chi \bigg[ \partial_i u_j (\partial_k u_j \partial_k u_i)^\epsilon - \partial_k u_i (\partial_k u_j \partial_i u_j)^\epsilon \bigg] dx dt \\
&+ \alpha^2 \int_0^T \int_{\mathbb{T}^3} \bigg[ u_j \partial_i \chi (\partial_k u_j \partial_k u_i)^\epsilon - u_i \partial_k \chi (\partial_k u_j \partial_i u_j)^\epsilon \bigg] dx dt \\
&+ \alpha^2 \int_0^T \int_{\mathbb{T}^3} \chi \bigg[ - \frac{1}{2} D_{3,\epsilon} (u)  - \frac{1}{2} \partial_i (u_i \partial_k u_j \partial_k u_j )^\epsilon + \frac{1}{2} u_i \partial_i (\partial_k u_j \partial_k u_j)^\epsilon \bigg]dx dt \\
&+ \alpha^2 \int_0^T \int_{\mathbb{T}^3} \bigg[ u_i  \partial_k u_j \partial_i \chi \partial_k u^\epsilon_j  + u_i  \partial_k u_j \partial_i u^\epsilon_j \partial_k \chi + u_i  \partial_k u_j u^\epsilon_j \partial_i \partial_k  \chi + \partial_k \partial_i \chi u_j  (u_i \partial_k u_j)^\epsilon \\
&+ \partial_k \chi \partial_i u_j  (u_i \partial_k u_j)^\epsilon\bigg] dx dt.
\end{align*}
\endgroup
The defect term that was used above is given by
\begingroup
\allowdisplaybreaks
\begin{align*}
D_{3,\epsilon} (x,t) = \int_{\mathbb{R}^3} &\big [ \partial_i \phi_\epsilon (\xi) \delta u_i (\xi ; x,t) \delta \partial_k u_j (\xi ;x,t) \delta \partial_k u_j (\xi ;x,t)  d \xi = - \partial_i (u_i \partial_k u_j \partial_k u_j)^\epsilon \\
&+ u_i \partial_i (\partial_k u_j \partial_k u_j)^\epsilon + 2 \partial_k u_j \partial_i (u_i \partial_k u_j)^\epsilon - 2u_i \partial_k u_j \partial_i \partial_k u_j^\epsilon \big ] d\xi .
\end{align*}
\endgroup
Now we can write that
\begin{align*}
&- \int_0^T \int_{\mathbb{T}^3 } u_j \partial_i u_j (u^\epsilon_i \chi) dx dt - \int_0^T \int_{\mathbb{T}^3} u_i \chi (u_j \partial_i u_j)^\epsilon dx dt \\
&= - \frac{1}{2} \int_0^T \int_{\mathbb{T}^3} \chi \partial_i \big[ u_i^\epsilon u_j  u_j \big] dx dt - \frac{1}{2} \int_0^T \int_{\mathbb{T}^3} \chi \partial_i \big[ u_i (u_j  u_j)^\epsilon \big] dx dt \\
&= \frac{1}{2} \bigg\langle \nabla \cdot \big( \lvert u \rvert^2 u^\epsilon + (u \cdot u)^\epsilon u \big), \chi \bigg\rangle.
\end{align*}
The next part we can write as
\begingroup
\allowdisplaybreaks
\begin{align*}
& \frac{1}{2} \alpha^2 \int_0^T \int_{\mathbb{T}^3 } \partial_k u_j \partial_k u_j \partial_i (u^\epsilon_i \chi) dx dt + \frac{1}{2} \alpha^2 \int_0^T \int_{\mathbb{T}^3} u_i (\partial_k u_j \partial_k u_j)^\epsilon \partial_i \chi dx dt \\
&= - \frac{1}{2} \alpha^2 \bigg\langle \nabla \cdot \big[ \partial_k u_j \partial_k u_j u^\epsilon + ( \partial_k u_j \partial_k u_j )^\epsilon u \big], \chi \bigg\rangle.
\end{align*}
\endgroup
The contribution of the pressure terms can subsequently be rewritten as
\begingroup
\allowdisplaybreaks
\begin{align*}
&\int_0^T \int_{\mathbb{T}^3} p \nabla \cdot (u^\epsilon \chi) dx dt - \int_0^T \int_{\mathbb{T}^3} u \chi \cdot \nabla p^\epsilon dx dt = \int_0^T \int_{\mathbb{T}^3} [p u^\epsilon + u p^\epsilon] \cdot \nabla \chi dx dt \\
&= - \bigg\langle \nabla \cdot (p u^\epsilon + p^\epsilon u), \chi \bigg\rangle.
\end{align*}
\endgroup
Combining all these derivations, we obtain the following equation
\begingroup
\allowdisplaybreaks
\begin{align*}
&\bigg\langle \partial_t ( u \cdot u^\epsilon) + \alpha^2 \partial_t (\nabla u : \nabla u^\epsilon) - \alpha^2 \partial_t \partial_i \bigg( u_j^\epsilon \partial_i u_j + u_j \partial_i u_j^\epsilon \bigg) + \alpha^2\partial_i \bigg( \partial_t u_j^\epsilon \partial_i u_j + \partial_t u_j \partial_i u_j^\epsilon \bigg) \\
&+ D_{2,\epsilon} (u) + \nabla \cdot ( (\lvert u \rvert^2 u)^\epsilon - (\lvert u \rvert^2)^\epsilon u ) + \nabla \cdot ( (u \cdot u^\epsilon) u) - \alpha^2 \partial_k u_i \partial_k u_j \partial_i u_j^\epsilon + \alpha^2 \partial_k u_j \partial_i u_j \partial_k u_i^\epsilon \\
&+ \alpha^2 \nabla \cdot ( \partial_k u_j u_j^\epsilon \partial_k u) - \alpha^2 \partial_k (\partial_k u_j \partial_i u_j u_i^\epsilon) - \alpha^2 \partial_i u_j (\partial_k u_j \partial_k u_i)^\epsilon + \alpha^2 \partial_k u_i (\partial_k u_j \partial_i u_j)^\epsilon\\
&+\alpha^2 \nabla \cdot ( u_j (\partial_k u_j \partial_k u)^\epsilon) - \alpha^2 \partial_k (u_i (\partial_k u_j \partial_i u_j)^\epsilon ) + \frac{1}{2} \alpha^2 D_{3,\epsilon} (u)  + \frac{1}{2} \alpha^2 \partial_i (u_i \partial_k u_j \partial_k u_j )^\epsilon \\
&- \frac{1}{2} \alpha^2 \partial_i (u_i (\partial_k u_j \partial_k u_j)^\epsilon ) + \alpha^2 \nabla \cdot (\partial_k u_j \partial_k u_j^\epsilon u) + \alpha^2 \partial_k (u_i \partial_k u_j \partial_i u_j^\epsilon) + \alpha^2 \partial_i \partial_k (u_i \partial_k u_j u_j^\epsilon)\\
&+ \alpha^2 \partial_k \partial_i ( u_j (u_i \partial_k u_j)^\epsilon) + \alpha^2 \partial_k ( \partial_i u_j (u_i \partial_k u_j)^\epsilon), \chi \bigg\rangle = 0.
\end{align*}
\endgroup
We observe that $D_{2,\epsilon} (u) $ and $D_{3,\epsilon} (u)$ are well-defined for any $\epsilon > 0$. It is now possible to  express the sum $\big[D_{2,\epsilon} (u) + \frac{1}{2} \alpha^2 D_{3,\epsilon} (u)\big]$ using the other terms in the above equation. Since the limit of all the other terms in the above equation is a well-defined as distributional  limit and is independent of the choice of mollifier, we conclude that the limit $\lim_{\epsilon \rightarrow 0} \big[ D_{2,\epsilon} (u) + \frac{1}{2} \alpha^2 D_{3,\epsilon} (u) \big]$ exists as a distribution and will be denoted by $\big[ D_2 (u) + \frac{1}{2} \alpha^2 D_3(u) \big]$. Then by taking the limit we obtain the following equation of local energy balance
\begingroup
\allowdisplaybreaks
\begin{align*}
&\bigg\langle \partial_t ( \lvert u \rvert^2) + \alpha^2 \partial_t ( \lvert \nabla u \rvert^2) - 2 \alpha^2 \partial_t \partial_i \bigg( u_j \partial_i u_j  \bigg) + 2 \alpha^2\partial_i \bigg( \partial_t u_j \partial_i u_j \bigg) + \big[ D_{2} (u) + \frac{1}{2} \alpha^2 D_{3} (u) \big]\\
&+ \nabla \cdot ( \lvert u \rvert^2 u)+ 2\alpha^2 \nabla \cdot ( \partial_k u_j u_j \partial_k u) + 2 \alpha^2 \nabla \cdot (\partial_k u_j \partial_k u_j u) + 2 \alpha^2 \partial_i \partial_k (u_i \partial_k u_j u_j), \chi \bigg\rangle = 0.
\end{align*}
\endgroup
In fact, we will prove, in what follows, that $D_2 (u) :=\lim_{\epsilon \rightarrow 0}  D_{2,\epsilon} (u)= 0$ for all weak solutions of the Euler-$\alpha$ equations, without any additional regularity assumption on $u$. This implies that $D_3(u):=\lim_{\epsilon \rightarrow 0} D_{3,\epsilon}$ exists as a distribution. For this reason we used a separate notation for  the terms $D_2$ and $D_3$  in the statement of the theorem.
\end{proof}
Now we state a sufficient condition for the defect terms, $D_2(u),D_3(u)$, to be zero.
\begin{proposition} \label{zerodefecteuler}
Assume that $u$ is a weak solution of the Euler-$\alpha$ equations such that it satisfies the following bounds
\begin{align*}
&\int_{\mathbb{T}^3} \lvert \delta u (\xi;x,t) \rvert^3 dx \leq C(t) \lvert \xi \rvert \sigma_2 (\lvert \xi \rvert), \\
&\int_{\mathbb{T}^3} \lvert \delta u (\xi;x,t) \rvert \lvert \delta \nabla u (\xi;x,t) \rvert^2 dx \leq C(t) \lvert \xi \rvert \sigma_3 (\lvert \xi \rvert),
\end{align*}
where $C \in L^1 (0,T)$ and $\sigma_k \in L^\infty_{\mathrm{loc}} (\mathbb{R})$ has the property that $\sigma_k (\lvert \xi \rvert) \rightarrow 0$ as $\lvert \xi \rvert \rightarrow 0$,  for $k=2,3$. Then $D_2 (u) = D_3 (u) = 0$.
\end{proposition}
\begin{proof}
The proof is completely analogous to the proof of Proposition \ref{zerodefect}.
\end{proof}
Now we state the regularity assumption that ensures that the sufficient conditions of Proposition \ref{zerodefecteuler} are satisfied.
\begin{proposition} \label{suffcondeuler}
Let $u$ be a weak solution of the Euler-$\alpha$ equations such that $u \in L^3 ((0,T); \linebreak B^s_{3,\infty} (\mathbb{T}^3))$ with $s > 1$. Then $D_2 (u) = D_3 (u) = 0$. In particular, the weak solution conserves energy.
\end{proposition}
\begin{proof}
We need to verify that the assumptions of Proposition \ref{zerodefecteuler} are satisfied. We first observe that by the Besov embedding $H^1 (\mathbb{T}^3) = B^1_{2,2} (\mathbb{T}^3) \subset B^{1/2}_{3,\infty} (\mathbb{T}^3)$ (see, e.g., \cite{leoni} for futher details). Therefore, by the above and \eqref{besovinequality} we have
\begin{equation*}
\int_{\mathbb{T}^3} \lvert \delta u (\xi;x,t) \rvert^3 dx \leq \lvert \xi \rvert^{3/2} \lVert u \rVert_{H^1 }^3.
\end{equation*}
We can then take $\sigma_2 (\lvert \xi \rvert) = \lvert \xi \rvert^{1/2}$, which indeed converges to zero as $\lvert \xi \rvert \to 0$. We therefore conclude that $D_2 (u) = 0$ for all weak solutions of the Euler-$\alpha$ equations, without any additional regularity assumption on $u$. The second assumption of Proposition \ref{zerodefecteuler} is also satisfied since
\begin{equation*}
\int_{\mathbb{T}^3} \lvert \delta u (\xi;x,t) \rvert \lvert \delta \nabla u (\xi;x,t) \rvert^2 dx \leq \lvert \xi \rvert^{2s - 1} \lVert u \rVert_{B^s_{3,\infty}}^3,
\end{equation*}
where we have applied inequality \eqref{besovinequality}. Since $s > 1$, then $\sigma_3 (\lvert \xi \rvert) = \lvert \xi \rvert^{2s-2}$ tends to zero as $\lvert \xi \rvert \to 0$.
\end{proof}
Finally, we prove the conservation of energy.
\begin{theorem}
Let $u$ be a weak solution of the Euler-$\alpha$ equations with $u\in L^3 ((0,T); B^s_{3,\infty} (\mathbb{T}^3))$ with $s > 1$. Then $u$ conserves energy, which means that
\begin{equation*}
\lVert u(\cdot, t_1) \rVert_{H^1} = \lVert u(\cdot, t_2) \rVert_{H^2},
\end{equation*}
for almost every $t_1, t_2 \in (0,T)$.
\end{theorem}
\begin{proof}
The proof works the same way as the proof of Theorem \ref{conservationtheorem}.
\end{proof}
\section{An overview of the results} \label{resultssection}

In the previous two sections we have proven sufficient conditions for weak solutions of the Leray-$\alpha$ model and the Euler-$\alpha$ equations to conserve energy. In this section we will provide an overview of the results for all the five models considered in this paper. The full proofs for the modified Leray-$\alpha$, Clark-$\alpha$ and Leray-$\alpha$ MHD models will be given in the appendices.
\begin{remark}[Besov Onsager exponents]
The sufficient conditions that we prove in this paper are stated in Table \ref{besovtable}. Notice that these are upper bounds on the Onsager exponents. The H\"older exponents which are the threshold for energy conservation are the same as the given Besov exponents, which can be seen by the Besov embedding $C^s (\mathbb{T}^3) \subset B^{s}_{3,\infty} (\mathbb{T}^3)$ for any $s > 0$ (see, e.g., \cite{leoni} for further details).

In the case of the MHD Leray-$\alpha$ model, we observe that it is possible to weaken the assumptions on $B$ by strengthening the assumptions on $v$. Therefore it is possible to `trade' regularity assumptions between the two quantities. This is also true for the standard MHD equations, as noted by \cite{buckmasterreview}.

\begin{table}[H]
\hspace*{0cm} \begin{tabular}{ |p{3.5cm}|p{3cm}|p{3cm}|p{5cm}|  }
 \hline
 \multicolumn{4}{|c|}{Onsager exponents (upper bounds)} \\
 \hline
 Model & Conserved quantity & Regularity assumption & Besov assumption  \\
 \hline
 Euler   & $\lVert v \rVert_{L^2}$ & $v \in L^\infty_t ( L^2_x)$ & $v \in L^3 ( (0,T);B^s_{3, \infty}), \; s>\frac{1}{3} $    \\
 Leray-$\alpha$ & $\lVert v \rVert_{L^2}$ & $v \in L^\infty_t ( L^2_x)$ & $v \in L^3 ( (0,T);B^s_{3, \infty}), \; s>0 $ \\
 Euler-$\alpha$ & $\lVert u \rVert_{H^1}$ & $u \in L^\infty_t ( H^1_x)$ & $u \in L^3 ( (0,T);B^s_{3, \infty}), \; s>1 $ \\
 Modified Leray-$\alpha$ & $\lVert u \rVert_{H^1}$ & $u \in L^\infty_t ( H^1_x)$ & $u \in L^3 ( (0,T);B^s_{3, \infty}), \; s>1 $  \\
 Clark-$\alpha$ & $\lVert u \rVert_{H^1}$ & $u \in L^\infty_t (H^1_x)$ & $u \in L^3 ( (0,T);B^s_{3, \infty}), \; s>1 $ \\
 MHD Leray-$\alpha$ & $\lVert v \rVert_{L^2}^2 + \lVert B \rVert_{L^2}^2$ & $v$, $B \in L^\infty_t (L^2_x)$ & $v \in L^3 ( (0,T);B^s_{3, \infty}),  s>0 $, \; $B \in L^3 ( (0,T);B^r_{3, \infty}), \; r>0 $ and $s + 2 r > 1$ \\
 \hline
\end{tabular}
\caption{Sufficient conditions for energy conservation stated in terms of Besov spaces (and H\"older spaces)}
\label{besovtable}
\end{table}
\end{remark}
It is also possible to find criteria in terms of Sobolev spaces, as we observe in the next remark.
\begin{remark}[Sobolev Onsager exponents] \label{onsagerexponents}
It should also be observed that for the models considered in this work one can also phrase the conditions in terms of Sobolev spaces or alternatively, as a condition for $v$ to have negative Sobolev regularity with a certain exponent. These results are stated in Table \ref{sobolevtable}.
\begin{table}[h]
\begin{tabular}{ |p{3.5cm}|p{3cm}|p{3cm}|p{2cm}|p{3cm}|  }
 \hline
 \multicolumn{5}{|c|}{Onsager exponents for Sobolev spaces in three dimensions} \\
 \hline
 Model & Conserved quantity & Standard regularity assumption & Condition on $u$  & Condition on $v$ \\
 \hline
 Euler & $\lVert v \rVert_{L^2}$ & $v \in L^\infty_t (L^2_x)$ &  $-$ & $H^{5/6}$ \\
 Leray-$\alpha$ & $\lVert v \rVert_{L^2}$ & $v \in L^\infty_t ( L^2_x)$ & $H^{5/2}$ &  $H^{1/2}$\\
Euler-$\alpha$ & $\lVert u \rVert_{H^1}$ & $u \in L^\infty_t ( H^1_x)$ & $H^{3/2}$ &  $H^{-1/2}$\\
 Modified Leray-$\alpha$ & $\lVert u \rVert_{H^1}$ & $u \in L^\infty_t ( H^1_x)$ & $H^{3/2}$  & $H^{-1/2}$ \\
 Clark-$\alpha$ & $\lVert u \rVert_{H^1}$ & $u \in L^\infty_t ( H^1_x)$ & $H^{3/2}$  & $H^{-1/2}$ \\
MHD Leray-$\alpha$ & $\lVert v \rVert_{L^2}^2 +\lVert B \rVert_{L^2}^2$ & $v,B \in L^\infty ( L^2)$ & $-$  & $v \in H^s, B \in H^r$, $s, r > 1/2$ and $s + 2 r > 5/2$ \\
 \hline
\end{tabular}
\caption{ The given Sobolev exponents are thresholds for energy conservation, i.e. if the solution has higher Sobolev regularity than given in this table, it must conserve energy.
}
\label{sobolevtable}
\end{table}
\end{remark}
\begin{remark}
The work \cite{beekiealpha} came to our attention while this paper was being written. It is useful to contrast the results of this paper with the results from \cite{beekiealpha}. In particular, the paper states the following condition for energy  conservation for the Euler-$\alpha$ equations
\begin{equation} \label{beekiecondition}
u \in L^3 ((0,T); B^s_{3, \infty} (\mathbb{T}^3)), \quad s > 1.
\end{equation}
Note that this condition is an only slightly stronger requirement compared to the one stated in Proposition \ref{zerodefecteuler}. In Proposition \ref{suffcondeuler} we have shown that condition \eqref{beekiecondition} can be derived from Proposition \ref{zerodefecteuler}.

It is important to emphasise that \cite{beekiealpha} only considers the Euler-$\alpha$ equations and not the 4 other models considered in this paper. The main focus of \cite{beekiealpha} is developing a convex integration scheme for the Euler-$\alpha$ equations to construct weak solutions that do not conserve energy. In this paper we provide a study of several subgrid scale $\alpha$-models and in particular investigate what determines the Onsager exponents for these models.

In addition, the paper \cite{beekiealpha} relies on the approach of \cite{constantin} by using commutator estimates to prove conservation of energy. We rely on the approach of \cite{duchon} by using an equation of energy balance and analysing the defect term. One advantage of the approach of using defect terms is that it is more straightforward to analyse them. In Remark \ref{fractionallaplacian} we will consider what happens to the Onsager exponent if the Helmholtz regularisation in equation \eqref{helmholtz} for the inviscid Leray-$\alpha$ model is replaced by $v = u + \alpha^{2 \theta} (-\Delta)^\theta u$. This leads to a linear relationship between the Onsager exponent and $\theta$. Deriving such a relation by using commutator estimates is possible, but would be more tedious.
\end{remark}

To conclude this section, we make a remark on what `determines' the Onsager exponent for a given PDE.
\begin{remark}[`Origin' of the Onsager exponent] \label{originremark}
Observe that the bound \eqref{leraycondition} in Proposition \ref{zerodefect} is of the form $|\xi|\sigma_1(|\xi|)$, where $\sigma_1=o (1)$, as $|\xi| \to 0$. This is the essence of why different models have different Onsager exponents. The required regularity is to ensure bounds on the defect terms of  the form $|\xi|\sigma_k(|\xi|)$, where the $|\xi|$  factor is `distributed' into fractional regularities among the three terms in the product, and  where  the relevant  $\sigma_k =o(1)$, as $|\xi|\to 0$. Therefore,  the more regular these terms are, the lower the Onsager exponent will be. For instance, for the case of the Euler equations one has the following condition for energy conservation \cite{duchon}
\begin{equation*}
\int_{\mathbb{T}^3} \lvert \delta v (\xi;x,t) \rvert^3 dx \leq C(t) \lvert \xi \rvert \sigma_1 ( \lvert \xi \rvert),
\end{equation*}
which implies an Onsager (Besov) exponent of $\frac{1}{3}$ since all the three terms are the same and hence the factor $\lvert \xi \rvert$ is equally distributed among the three terms. In the case of the Leray-$\alpha$ model, the term $\delta v $ has two degrees less regularity than the term $\delta u$. Therefore assuming a Besov exponent of more than $0$ for $v$ implies that $ u$ has a Besov exponent of more than $2$ which combines to give a total exponent of more than 1.

So the crucial difference between the Euler equations and the Leray-$\alpha$ model is that in the defect term one factor of $\delta v$ is replaced by a factor of $\delta u$, which has two degrees more regularity. Therefore under less assumptions on the weak solution than for the Euler equations, the regularity of the factor $\delta u$ still can ensure that $\sigma_1=o(1)$, as $\lvert \xi \rvert \to 0$, which implies that the defect term is zero.
\end{remark}

\section{Conclusion} \label{conclusion}
In this paper we have studied several subgrid scale $\alpha$-models of turbulence, obtained equations of local energy balance and found Onsager exponents in terms of Besov, H\"older and Sobolev spaces. Several of these exponents were different from $1/3$, which is the Onsager exponent for the Euler equations. This means that the Onsager exponent is not universal and is plausibly determined by the regularity of the coefficient functions of the nonlinearity.

It was already discussed in Remark \ref{originremark} what mechanism seems to lead to different values of the Onsager exponent for the different models. We mention that it is straightforward to predict the Onsager exponent for a given PDE of this type without doing a complete analysis by establishing an equation of local energy balance with a precise defect term.

If a PDE of this type has a formally conserved quantity, the way one obtains the conservation law is by multiplying by a given function (such as $u$ or $v$) and then invoking the divergence theorem to show the time derivative of the conserved quantity is zero.

This already tells you what kind of products the defect term will contain. From Lemma \ref{zerodefect} one can conclude that the function $\sigma$ in equation \eqref{leraycondition} needs to be $o(1)$. Therefore one can already tell which regularity assumption will assure a condition of type \eqref{leraycondition} is satisfied. The difference in Onsager exponents therefore is caused by the different forms of the defect terms.

For the Leray-$\alpha$ model one has a product $\lvert \delta v \rvert^2 \lvert \delta u \rvert$. For the Euler-$\alpha$ equations it is $\lvert \delta \nabla u \rvert^2 \lvert \delta u \rvert$. Since $u$ is more regular than $v$, it follows that the Euler-$\alpha$ model has a lower Onsager exponent than the Leray-$\alpha$ model (namely a Besov exponent of 1 instead of $2$ for $u$).

Several other MHD-$\alpha$ models exist and a similar analysis for these should be possible. Examples include the modified Leray-$\alpha$ MHD model, see \cite{linshiz} for more details. In addition, one can consider what happens if one uses a different regularisation instead of the standard Helmholtz regularisation which was used throughout the paper (see also \cite{olson}). In the following remark we look at this point for the Leray-$\alpha$ model.
\begin{remark} \label{fractionallaplacian}
A possibly interesting generalisation of the results in this work is the following: Suppose that instead of the Helmholtz regularisation $v = u - \alpha^2 \Delta u$ we consider the regularisation
\begin{equation*}
v = u + \alpha^{2\theta} (-\Delta)^\theta u,
\end{equation*}
for $\theta > 0$ positive. For such a regularisation $v(\cdot, t) \in L^2 (\mathbb{T}^3)$ implies that $u(\cdot, t) \in H^{2 \theta} (\mathbb{T}^3)$. Note that a general case involving viscosity given by a fractional Laplacian was studied in \cite{olson}. Assume in addition that $u$ and $v$ are incompressible and satisfy the equation
\begin{equation*}
\partial_t v + \nabla \cdot (u \otimes v) + \nabla p = 0.
\end{equation*}
For the choice $\theta = 1$ we recover the Leray-$\alpha$ model. One can then ask how the Onsager exponent changes as $\theta \rightarrow 0$. In the formal limit we recover the Euler equations. That means that one would expect the Onsager exponent of this model to converge to $\frac{1}{3}$ as $\theta \rightarrow 0$ (formally at least). By similar reasoning as in the proof of Theorem \ref{energyequationtheorem} we can establish an equation of local energy balance. One can make the following observations.
\begin{itemize}
    \item If $\theta > \frac{1}{2}$ the Onsager exponent is at most 0 (we are referring to the H\"older exponent here). This is because $v (\cdot, t) \in C^{0+} (\mathbb{T}^3)$ implies that $u (\cdot, t) \in C^{0, 2 \theta+} (\mathbb{T}^3)$, here the notation $u (\cdot, t) \in C^{0,\beta+} (\mathbb{T}^3)$ means that $u(\cdot, t) \in C^{0,\beta'} (\mathbb{T}^3)$ for some $\beta' > \beta$. Therefore the fact that $v (\cdot, t) \in C^{0+} (\mathbb{T}^3)$ implies that the condition from Lemma \ref{zerodefect} is satisfied and the energy is conserved by Theorem \ref{conservationtheorem}.
    \item If $0 < \theta < \frac{1}{2}$ then we can deduce the Onsager exponent as follows. Suppose $v(\cdot, t) \in C^{0, \gamma} (\mathbb{T}^3)$, then we have that $u(\cdot, t) \in C^{0, \gamma + 2 \theta} (\mathbb{T}^3)$. Together the product $\lvert \delta u \rvert \lvert \delta v \rvert^2$ must have a little bit more than Lipschitz decay as $\lvert \xi \rvert \rightarrow 0$, this leads to the equation
    \begin{equation*}
    3 \gamma + 2 \theta = 1 \implies \gamma = \frac{1}{3} - \frac{2}{3} \theta.
    \end{equation*}
    This means that if $\gamma > \frac{1}{3} - \frac{2}{3} \theta$, energy is conserved by Theorem \ref{conservationtheorem}. Like in the case of the original Onsager exponent for the Euler equations, this inequality is strict. Observe that in the limit $\theta \rightarrow 0$ we indeed recover the Onsager exponent $\frac{1}{3}$ for the Euler equations. Hence we have justified that the Onsager exponent indeed converges formally to $\frac{1}{3}$ as $\theta \rightarrow 0$.
\end{itemize}
\end{remark}
\section*{Acknowledgements}
The first author would like to thank the Groningen University Fund, Wolfson College Cambridge, the Cambridge Trust and the Hendrik Muller fund for financial support when this work was completed. The authors would like to thank the Isaac Newton Institute for Mathematical Sciences, Cambridge, for support and hospitality during the programme ``Mathematical aspects of turbulence: where do we stand?'' where work on this paper was undertaken. This work was supported by EPSRC grant no. EP/K032208/1.

\begin{appendices}
\section{The Modified Leray-$\alpha$ model} \label{modifiedleraysection}
The Modified Leray-$\alpha$ model is given by
\begin{equation*}
\partial_t v  +  \nabla \cdot (v \otimes u) + \nabla p = 0, \quad v = u - \alpha^2 \Delta u, \quad \nabla \cdot u = \nabla \cdot v = 0.
\end{equation*}
The formally conserved quantity is again $\lVert u \rVert_{H^1}$.
\begin{definition}
We say that $u \in L^\infty ((0,T); H^1 (\mathbb{T}^3))$ and $p \in L^\infty ((0,T); L^1 (\mathbb{T}^3))$ is a weak solution of the Modified Leray-$\alpha$ model if it satisfies the following equations for all $\psi \in \mathcal{D } (\mathbb{T}^3 \times (0,T); \mathbb{R}^3)$ and $\chi \in \mathcal{D} (\mathbb{T}^3 \times (0,T); \mathbb{R})$
\begingroup
\allowdisplaybreaks
\begin{align*}
&\int_0^T \int_{\mathbb{T}^3 } u \cdot \partial_t \psi dx dt + \alpha^2 \int_0^T \int_{\mathbb{T}^3 } \nabla u : \nabla \partial_t \psi dx dt + \int_0^T \int_{\mathbb{T}^3}  u \otimes u : \nabla \psi dx dt \\
&+ \alpha^2 \int_0^T \int_{\mathbb{T}^3} \partial_k u \otimes \partial_k u : \nabla \psi dx dt + \alpha^2 \int_0^T \int_{\mathbb{T}^3} \partial_k u \otimes  u : \nabla \partial_k \psi dx dt + \int_0^T \int_{\mathbb{T}^3} p \nabla \cdot \psi dx dt = 0,  \\
&\int_0^T \int_{\mathbb{T}^3} u_i \partial_i \chi dx dt = 0.
\end{align*}
\endgroup
The pressure is determined up to a constant, which we fix by
\begin{equation*}
\int_{\mathbb{T}^3} p(x,t) dx = 0.
\end{equation*}
\end{definition}
As before, we now state a result stating that we can take a larger space of test functions.
\begin{lemma}\label{gentestfuncmod}
The weak formulation of the Modified Leray-$\alpha$ model still holds for test functions $\psi \in W^{1,1}_0 ((0,T) ; H^1 (\mathbb{T}^3)) \cap L^1 ((0,T); H^3 (\mathbb{T}^3))$
\end{lemma}
\begin{proof}
The proof proceeds in the same fashion as the proof of Lemma \ref{generaltestfunctions}. Again one relies on a limiting argument by using the approximation of Sobolev functions by smooth functions.
\end{proof}
Now we establish the equation of local energy balance.
\begin{theorem}
Let $u$ be a weak solution of the Modified Leray-$\alpha$ model, such that $u \in L^3 ((0,T); W^{1,9/4} (\mathbb{T}^3))$. Then the equation of local energy balance holds
\begin{align*}
&\partial_t ( \lvert u \rvert^2 ) + \alpha^2 \partial_t (\lvert \nabla u \rvert^2) - 2 \alpha^2 \partial_t \partial_i \bigg( u_j \partial_i u_j\bigg) + 2 \alpha^2\partial_i \bigg( \partial_t u_j \partial_i u_j \bigg) + 2 \nabla \cdot (pu) \\
&+ D_{4} (u) + \nabla \cdot (\lvert u\rvert^2 u) + 2\nabla \cdot \big( u_j \partial_k u_j \partial_k u  \big) + \nabla \cdot ( \partial_k (\lvert u \rvert^2 \partial_k u) ) + D_{5} (u) = 0. \nonumber
\end{align*}
This equation holds in the sense of distributions. Here the defect terms are given by
\begin{align*}
D_{4} (u) &\coloneqq \lim_{\epsilon \rightarrow 0} \int_{\mathbb{R}^3} \nabla_\xi \phi_\epsilon (\xi) \cdot \delta u(\xi;x,t) \lvert \delta u (\xi ;x,t) \lvert^2 d \xi, \\
D_{5} (u) &\coloneqq \lim_{\epsilon \rightarrow 0} \int_{\mathbb{R}^3} \partial_i \phi_\epsilon (\xi) \delta u_j (\xi;x,t) \delta \partial_k u_i (\xi;x,t) \delta \partial_k u_j (\xi;x,t) d \xi.
\end{align*}
Note that the limits hold in the sense of distributions and are independent of the choice of mollifier.
\end{theorem}
\begin{proof}
Mollifying the equation of the Modified Leray-$\alpha$ model in space (with $\phi_\epsilon$) gives us that $\partial_t v^\epsilon \in L^\infty ((0,T); C^\infty (\mathbb{T}^3))$. Therefore we can conclude that $u^\epsilon, v^\epsilon \in W^{1,\infty} ((0,T); C^\infty (\mathbb{T}^3))$.

By Lemma \ref{gentestfuncmod} we can take $u^\epsilon \chi$ (for $\chi \in \mathcal{D} (\mathbb{T}^3 \times (0,T); \mathbb{R})$ ) as a test function in the weak formulation. If we subtract the mollified equation multiplied by $u \chi$ from this equation, we obtain that
\begingroup
\allowdisplaybreaks
\begin{align*}
&\int_0^T \int_{\mathbb{T}^3 } u \partial_t (u^\epsilon \chi) dx dt - \int_0^T \int_{\mathbb{T}^3} u \chi \partial_t (u^\epsilon) dx dt + \alpha^2 \int_0^T \int_{\mathbb{T}^3 } \nabla u : \nabla \partial_t (u^\epsilon \chi) dx dt  \\
&- \alpha^2 \int_0^T \int_{\mathbb{T}^3} \nabla (u \chi) :  \partial_t \nabla u^\epsilon dx dt +\int_0^T \int_{\mathbb{T}^3 } (u \otimes u) : \nabla (u^\epsilon \chi) dx dt \\
&- \int_0^T \int_{\mathbb{T}^3 }  \chi u \cdot  (\nabla \cdot (u \otimes u)^\epsilon) dx dt + \alpha^2 \int_0^T \int_{\mathbb{T}^3} \partial_k u \otimes \partial_k u : \nabla (u^\epsilon \chi ) -\\
&\alpha^2 \int_0^T \int_{\mathbb{T}^3} u \chi \cdot ( \nabla \cdot (\partial_k u \otimes \partial_k u)^\epsilon ) dx dt + \alpha^2 \int_0^T \int_{\mathbb{T}^3} \partial_k u \otimes u: \nabla \partial_k \chi dx dt \\
&- \alpha^2 \int_0^T \int_{\mathbb{T}^3} \partial_k (u \chi) \cdot (\nabla \cdot (\partial_k u \otimes u)^\epsilon) + \int_0^T \int_{\mathbb{T}^3} p \nabla \cdot (u^\epsilon \chi) dx dt - \int_0^T \int_{\mathbb{T}^3} u \chi \cdot \nabla p^\epsilon dx dt = 0.
\end{align*}
\endgroup
One can prove that $p \in L^{3/2} ((0,T); L^{9/8} (\mathbb{T}^3))$ and $\partial_t v \in L^{3/2} ((0,T); W^{-2,9/5} (\mathbb{T}^3))$, which implies that $\partial_t u \in L^{3/2} ((0,T); L^{9/5} (\mathbb{T}^3))$. The time derivative terms work the same way as before, we can write that
\begingroup
\allowdisplaybreaks
\begin{align*}
&\int_0^T \int_{\mathbb{T}^3 } u \partial_t (u^\epsilon \chi) dx dt - \int_0^T \int_{\mathbb{T}^3} u \chi \partial_t (u^\epsilon) dx dt + \alpha^2 \int_0^T \int_{\mathbb{T}^3 } \nabla u : \nabla \partial_t (u^\epsilon \chi) dx dt  \\
&- \alpha^2 \int_0^T \int_{\mathbb{T}^3} \nabla (u \chi) :  \partial_t \nabla u^\epsilon dx dt = \int_0^T \int_{\mathbb{T}^3} \bigg[ (u \cdot u^\epsilon) \partial_t \chi + \alpha^2 \nabla u : \nabla u^\epsilon \partial_t \chi \bigg] dx dt \\
&+ \alpha^2 \int_0^T \int_{\mathbb{T}^3} \bigg[ \nabla u : \nabla \chi \otimes \partial_t u^\epsilon + \nabla u : \partial_t \nabla \chi \otimes u^\epsilon - \nabla \chi \otimes u : \partial_t \nabla u^\epsilon \bigg] dx dt \\
&= \int_0^T \int_{\mathbb{T}^3} \bigg[ (u \cdot u^\epsilon) \partial_t \chi + \alpha^2 \nabla u : \nabla u^\epsilon \partial_t \chi \bigg] dx dt \\
&+ \alpha^2 \int_0^T \int_{\mathbb{T}^3} \bigg[ \nabla u : \nabla \chi \otimes \partial_t u^\epsilon + \nabla u : \partial_t \nabla \chi \otimes u^\epsilon + \partial_t \nabla \chi \otimes u : \nabla u^\epsilon + \nabla \chi \otimes \partial_t  u :\nabla u^\epsilon \bigg] dx dt.
\end{align*}
\endgroup
This can be written as
\begin{equation*}
- \bigg\langle \partial_t ( u \cdot u^\epsilon) + \alpha^2 \partial_t (\nabla u : \nabla u^\epsilon) + \alpha^2 \partial_t \partial_i \bigg( u_j^\epsilon \partial_i u_j + u_j \partial_i u_j^\epsilon \bigg) + \alpha^2\partial_i \bigg( \partial_t u_j^\epsilon \partial_i u_j + \partial_t u_j \partial_i u_j^\epsilon \bigg) , \chi \bigg\rangle.
\end{equation*}
In the limit $\epsilon \rightarrow 0$ this converges to
\begin{equation*}
- \bigg\langle \partial_t ( \lvert u \rvert^2 ) + \alpha^2 \partial_t (\lvert \nabla u \rvert^2) + 2 \alpha^2 \partial_t \partial_i \bigg( u_j \partial_i u_j\bigg) + 2 \alpha^2\partial_i \bigg( \partial_t u_j \partial_i u_j \bigg), \chi \bigg\rangle.
\end{equation*}
The pressure terms can be written as
\begin{align*}
&\int_0^T \int_{\mathbb{T}^3} p \nabla \cdot (u^\epsilon \chi) dx dt - \int_0^T \int_{\mathbb{T}^3} u \chi \cdot \nabla p^\epsilon dx dt = \int_0^T \int_{\mathbb{T}^3} [p u^\epsilon + u p^\epsilon] \cdot \nabla \chi dx dt \\
&= - \bigg\langle \nabla \cdot (p u^\epsilon + p^\epsilon u ), \chi \bigg\rangle.
\end{align*}
The most involved part are the advective (cubic in $u$) terms. We introduce the defect term given by
\begin{equation*}
D_{4,\epsilon} (u) (x,t) \coloneqq \frac{1}{2} \int_{\mathbb{R}^3} \nabla_\xi \phi_\epsilon (\xi) \cdot \delta u(\xi;x,t) \lvert \delta u (\xi ;x,t) \lvert^2 d \xi,
\end{equation*}
As before we can write that
\begingroup
\allowdisplaybreaks
\begin{align*}
&\int_0^T \int_{\mathbb{T}^3 } (u \otimes u) : \nabla (u^\epsilon \chi) dx dt - \int_0^T \int_{\mathbb{T}^3 }  \chi u \cdot  (\nabla \cdot (u \otimes u)^\epsilon) dx dt \\
&= \int_0^T \int_{\mathbb{T}^3} \bigg[- \chi D_{4, \epsilon} (u) (x,t) - \frac{1}{2} (\lvert u \rvert^2 u)^\epsilon \cdot \nabla \chi + \frac{1}{2} u (\lvert u \rvert^2)^\epsilon \cdot \nabla \chi \bigg] dx dt \\
&+ \int_0^T \int_{\mathbb{T}^3 } (u \otimes u) : \nabla \chi \otimes u^\epsilon dx dt,
\end{align*}
\endgroup
and observe that as $\epsilon \rightarrow 0$ this converges to
\begin{equation*}
\int_0^T \int_{\mathbb{T}^3 } (u \otimes u) : \nabla \chi \otimes u dx dt.
\end{equation*}
The other advective terms can be written as (omitting the factor $\alpha^2$)
\begingroup
\allowdisplaybreaks
\begin{align*}
& \int_0^T \int_{\mathbb{T}^3} \partial_k u_i \partial_k u_j \partial_i (u^\epsilon_j \chi ) dx dt - \int_0^T \int_{\mathbb{T}^3} u_j \chi \partial_i (\partial_k u_i  \partial_k u_j)^\epsilon  dx dt \\
&+ \int_0^T \int_{\mathbb{T}^3} \partial_k u_i u_j \partial_i \partial_k (u^\epsilon_j \chi) dx dt - \int_0^T \int_{\mathbb{T}^3} \partial_k (u_j \chi)  \partial_i  (\partial_k u_i  u_j)^\epsilon \\
&= \int_0^T \int_{\mathbb{T}^3} \partial_k u_i \partial_k u_j u^\epsilon_j \partial_i \chi  dx dt + \int_0^T \int_{\mathbb{T}^3} \partial_k u_i u_j \partial_k (u^\epsilon_j \partial_i \chi) dx dt \\
&+ \int_0^T \int_{\mathbb{T}^3} \partial_k u_i u_j \partial_i u^\epsilon_j \partial_k \chi dx dt -  \int_0^T \int_{\mathbb{T}^3} u_j \partial_k \chi  \partial_i  (\partial_k u_i  u_j)^\epsilon \\
&+ \int_0^T \int_{\mathbb{T}^3} \partial_k u_i \partial_k u_j \partial_i u^\epsilon_j \chi  dx dt - \int_0^T \int_{\mathbb{T}^3} u_j \chi \partial_i (\partial_k u_i  \partial_k u_j)^\epsilon  dx dt \\
&+ \int_0^T \int_{\mathbb{T}^3} \partial_k u_i u_j \partial_i \partial_k u^\epsilon_j \chi dx dt - \int_0^T \int_{\mathbb{T}^3} \partial_k u_j \chi  \partial_i  (\partial_k u_i  u_j)^\epsilon.
\end{align*}
\endgroup
Now we can check that
\begingroup
\allowdisplaybreaks
\begin{align*}
&\int_0^T \int_{\mathbb{T}^3} \partial_k u_i u_j u^\epsilon_j \partial_i \partial_k \chi dx dt + \int_0^T \int_{\mathbb{T}^3} \partial_k u_i u_j \partial_i u^\epsilon_j \partial_k \chi dx dt -  \int_0^T \int_{\mathbb{T}^3} u_j \partial_k \chi  \partial_i  (\partial_k u_i  u_j)^\epsilon \\
&= \int_0^T \int_{\mathbb{T}^3} \partial_k u_i u_j u^\epsilon_j \partial_i \partial_k \chi dx dt + \int_0^T \int_{\mathbb{T}^3} \partial_k u_i u_j \partial_i u^\epsilon_j \partial_k \chi dx dt +  \int_0^T \int_{\mathbb{T}^3} \partial_i u_j \partial_k \chi  (\partial_k u_i  u_j)^\epsilon \\
&+  \int_0^T \int_{\mathbb{T}^3} u_j \partial_k \partial_i \chi    (\partial_k u_i  u_j)^\epsilon \xrightarrow[]{\epsilon \rightarrow 0} \int_0^T \int_{\mathbb{T}^3} \partial_k u_i \lvert u \rvert^2 \partial_i \partial_k \chi dx dt \\
&= -\bigg\langle \nabla \cdot ( \partial_k (\lvert u \rvert^2 \partial_k u) ), \chi \bigg\rangle.
\end{align*}
\endgroup
We have more divergence terms, which are given by
\begingroup
\allowdisplaybreaks
\begin{align*}
&\int_0^T \int_{\mathbb{T}^3} \partial_k u_i \partial_k u_j u^\epsilon_j \partial_i \chi  dx dt + \int_0^T \int_{\mathbb{T}^3} \partial_k u_i u_j  \partial_k u^\epsilon_j \partial_i \chi dx dt \\
&= - \bigg\langle \nabla \cdot \big( u_j^\epsilon \partial_k u_j \partial_k u +u_j \partial_k u_j^\epsilon \partial_k u \big), \chi \bigg\rangle.
\end{align*}
\endgroup
We compute a new defect term which is given by
\begin{align*}
D_{5,\epsilon} (u) &= \int_{\mathbb{R}^3} \partial_i \phi_\epsilon (\xi) \delta u_j (\xi;x,t) \delta \partial_k u_i (\xi;x,t) \delta \partial_k u_j (\xi;x,t) d \xi = - \partial_i ( u_j \partial_k u_j \partial_k u_i)^\epsilon + \partial_k u_i \partial_i (u_j \partial_k u_j)^\epsilon \\
&+ \partial_k u_j \partial_i (u_j \partial_k u_i)^\epsilon + u_j \partial_i (\partial_k u_j \partial_k u_i)^\epsilon - \partial_k u_j \partial_k u_i \partial_i u_j^\epsilon - u_j \partial_k u_i \partial_k \partial_i u_j^\epsilon.
\end{align*}
Finally, we are left with
\begingroup
\allowdisplaybreaks
\begin{align*}
&\int_0^T \int_{\mathbb{T}^3} \partial_k u_i \partial_k u_j \partial_i u^\epsilon_j \chi  dx dt - \int_0^T \int_{\mathbb{T}^3} u_j \chi \partial_i (\partial_k u_i  \partial_k u_j)^\epsilon  dx dt \\
&+ \int_0^T \int_{\mathbb{T}^3} \partial_k u_i u_j \partial_i \partial_k u^\epsilon_j \chi dx dt - \int_0^T \int_{\mathbb{T}^3} \partial_k u_j \chi  \partial_i  (\partial_k u_i  u_j)^\epsilon dx dt \\
&= \int_0^T \int_{\mathbb{T}^3} \chi \bigg[ - D_{5,\epsilon} (u) - \partial_i ( u_j \partial_k u_j \partial_k u_i)^\epsilon + \partial_k u_i \partial_i (u_j \partial_k u_j)^\epsilon \bigg] dx dt \\
&= - \bigg\langle D_{5,\epsilon} (u) - \nabla \cdot (u_j \partial_k u_j \partial_k u)^\epsilon + \nabla \cdot (\partial_k u (u_j \partial_k u_j)^\epsilon ), \chi \bigg\rangle.
\end{align*}
\endgroup
Now we put everything together and write down the equation of local energy balance which is given by (in the limit as $\epsilon \rightarrow 0$)
\begin{align*}
&\bigg\langle \partial_t ( \lvert u \rvert^2 ) + \alpha^2 \partial_t (\lvert \nabla u \rvert^2) - 2 \alpha^2 \partial_t \partial_i \bigg( u_j \partial_i u_j\bigg) + 2 \alpha^2\partial_i \bigg( \partial_t u_j \partial_i u_j \bigg) + 2 \nabla \cdot (pu) + D_{4} (u) + \nabla \cdot (\lvert u\rvert^2 u) \\
&+ 2\nabla \cdot \big( u_j \partial_k u_j \partial_k u  \big) + \nabla \cdot ( \partial_k (\lvert u \rvert^2 \partial_k u) ) + D_{5} (u) , \chi \bigg\rangle = 0.
\end{align*}
We can argue as before that the limiting distribution $D_4 (u) + D_5 (u)$ is independent of the choice of mollifier. Like in the case for the Euler-$\alpha$ equations, we will show that $D_4 (u) = 0$ for all weak solutions of the modified Leray-$\alpha$ model. Therefore $\lim_{\epsilon \rightarrow 0} D_{5,\epsilon}$ makes sense as a distribution.
\end{proof}
Now we state the usual sufficient conditions for the defect terms to be zero.
\begin{proposition} \label{sufficientmodified}
Let $u$ be a weak solution of the inviscid modified Leray-$\alpha$ model such that $u \in L^\infty ((0,T); W^{1,9/4} (\mathbb{T}^3))$. Moreover, we assume that
\begingroup
\allowdisplaybreaks
\begin{align*}
\int_{\mathbb{T}^3} \lvert \delta u (\xi;x,t) \rvert^3 dx &\leq C(t) \lvert \xi \rvert \sigma_4 (\lvert \xi \rvert), \\
\int_{\mathbb{T}^3} \lvert \delta u (\xi;x,t) \rvert \lvert \delta \nabla u (\xi;x,t) \rvert^2 dx &\leq C(t) \lvert \xi \rvert \sigma_5 (\lvert \xi \rvert),
\end{align*}
\endgroup
where $C \in L^1 (0,T)$ and $\sigma_k \in L^\infty_{\mathrm{loc}} (\mathbb{R})$ such that $\sigma_k (\lvert \xi \rvert) \rightarrow 0$ as $\lvert \xi \rvert \rightarrow 0$, for $k=4,5$. Then $D_4 (u) = D_5 (u) = 0$. In particular, the weak solution conserves the energy.
\end{proposition}
\begin{proof}
The proof proceeds in the same fashion as the proof of Proposition \ref{zerodefect}.
\end{proof}
Now we formulate sufficient conditions for the defect term to be zero.
\begin{proposition}
Let $u$ be a weak solution of the modified Leray-$\alpha$ model, such that $u \in L^3 ((0,T); B^{s}_{3,\infty} (\mathbb{T}^3))$ with $s > 1$, then $D_4 (u) = D_5 (u) = 0$, which implies that energy is conserved.
\end{proposition}
\begin{proof}
We need to prove that the conditions of Proposition \ref{sufficientmodified} are satisfied. We observe that the first inequality in Proposition \ref{sufficientmodified} is already satisfied for any weak solution of the model. The second inequality can be observed to hold because
\begin{equation*}
\int_{\mathbb{T}^3} \lvert \delta u (\xi;x,t) \rvert \lvert \delta \nabla u (\xi;x,t) \rvert^2 dx \leq \lvert \xi \rvert^{3s - 2} \lVert u \rVert_{B^s_{3,\infty}}^3.
\end{equation*}
Note that in the above we have used inequality \eqref{besovinequality}.
\end{proof}
Now we can conclude that energy is indeed conserved.
\begin{theorem}
Let $u \in L^3 ((0,T); B^s_{3,\infty} (\mathbb{T}^3))$ with $s > 1$ be a weak solution of the modified Leray-$\alpha$ model, then the energy is conserved, which means that
\begin{equation*}
\lVert u(\cdot, t_1) \rVert_{H^1} = \lVert u(\cdot, t_2) \rVert_{H^1},
\end{equation*}
for almost all times $t_1, t_2 \in (0,T)$.
\end{theorem}
\begin{proof}
The proof goes the same way as the proof of Theorem \ref{conservationtheorem}.
\end{proof}
\section{The Clark-$\alpha$ model} \label{clarksection}
The equations for the Clark-$\alpha$ model are
\begin{align*}
\partial_t v  +  \nabla \cdot (u \otimes v) &+ \nabla \cdot (v \otimes u) - \nabla \cdot (u \otimes u) - \alpha^2 \nabla \cdot (\nabla u \cdot \nabla u^T ) + \nabla p = 0, \nonumber \\
&v = u - \alpha^2 \Delta u, \quad \nabla \cdot u = \nabla \cdot v = 0.
\end{align*}
The conserved quantity is the norm $\lVert u \rVert_{H^1}$.
\begin{definition}
We call $u \in L^\infty ((0,T); H^1 (\mathbb{T}^3))$ and $p \in L^\infty ((0,T); L^1 (\mathbb{T}^3))$ a weak solution of the Clark-$\alpha$ model if for all $\psi \in \mathcal{D} (\mathbb{T}^3 \times (0,T); \mathbb{R}^3)$ and $\chi \in \mathcal{D} (\mathbb{T}^3 \times (0,T); \mathbb{R})$ satisfies the following equations
\begingroup
\allowdisplaybreaks
\begin{align}
&\int_0^T \int_{\mathbb{T}^3 } u \cdot \partial_t \psi dx dt + \alpha^2 \int_0^T \int_{\mathbb{T}^3 } \nabla u : \nabla \partial_t \psi dx dt + \int_0^T \int_{\mathbb{T}^3 } (u \otimes u) : \nabla \psi dx dt \nonumber \\
&+2\alpha^2 \int_0^T \int_{\mathbb{T}^3} \partial_k u \otimes \partial_k u : \nabla \psi dx dt + \alpha^2 \int_0^T \int_{\mathbb{T}^3} u \otimes \partial_k u: \nabla \partial_k \psi dx dt  \nonumber \\
& + \alpha^2 \int_0^T \int_{\mathbb{T}^3} \partial_k u \otimes u : \nabla \partial_k  \psi dx dt - \alpha^2 \int_0^T \int_{\mathbb{T}^3} \nabla u \cdot (\nabla u)^T : \nabla \psi dx dt \nonumber \\
& + \int_0^T \int_{\mathbb{T}^3} p \nabla \cdot \psi dx dt = 0, \nonumber \\
&\int_0^T \int_{\mathbb{T}^3} u_i \partial_i \chi dx dt = 0. \nonumber
\end{align}
\endgroup
The pressure is determined by
\begin{equation*}
\int_{\mathbb{T}^3} p(x,t) dx = 0.
\end{equation*}
\end{definition}
Now we extend the space of test functions.
\begin{lemma} \label{genclark}
The weak formulation of the Clark-$\alpha$ model still holds for test functions \\  $\psi \in W^{1,1}_0 ((0,T);  H^1 (\mathbb{T}^3)) \cap L^1 ((0,T); H^3 (\mathbb{T}^3))$.
\end{lemma}
\begin{proof}
The proof is the same as for the other models considered in this paper.
\end{proof}
Now we establish the equation of local energy balance.
\begin{theorem}
Let $u$ be a weak solution of the Clark-$\alpha$ model such that $u \in L^3 ((0,T); \linebreak W^{1,9/4} (\mathbb{T}^3))$. Then the following equation of local energy balance holds
\begin{align*}
&\partial_t (\lvert u \rvert^2) + \alpha^2 \partial_t (\lvert \nabla u \rvert^2) - 2 \alpha^2 \partial_t \partial_i \bigg( u_j \partial_i u_j \bigg) + 2 \alpha^2\partial_i \bigg( \partial_t u_j \partial_i u_j  \bigg)   + \nabla \cdot ( \lvert u \rvert^2 u) \\
&+ 2 \alpha^2 \partial_i ( u_j \partial_k u_j \partial_k u_i) + \alpha^2 \partial_i \partial_k (u_j \partial_k u_i u_j) + 3\alpha^2 \partial_k (\partial_i u_j  \partial_k u_i u_j) + \bigg[D_{6} (u) + \alpha^2 D_{7} (u) \\
&+ \frac{1}{2} \alpha^2 D_{8} (u)\bigg] + 2\nabla \cdot (p u ) = 0,
\end{align*}
which holds in the sense of distributions. The defect terms are given by
\begingroup
\allowdisplaybreaks
\begin{align*}
&\bigg[D_{6} (u) + \alpha^2 D_{7} (u) + \frac{1}{2} \alpha^2 D_{8} (u)\bigg] (x,t) \coloneqq \lim_{\epsilon \rightarrow 0} \bigg[ \frac{1}{2} \int_{\mathbb{R}^3} \nabla_\xi \phi_\epsilon (\xi) \cdot \delta u(\xi;x,t) \lvert \delta u (\xi ;x,t) \lvert^2 d \xi \\
&+ \alpha^2 \int_{\mathbb{R}^3} \partial_i \phi_\epsilon (\xi) \delta u_j (\xi;x,t) \delta \partial_k u_i (\xi;x,t) \delta \partial_k u_j (\xi;x,t) d \xi \\
&+ \frac{1}{2} \alpha^2 \int_{\mathbb{R}^3} \partial_i \phi_\epsilon (\xi) \delta u_i (\xi ; x,t) \lvert \delta \partial_k u_j (\xi ;x,t) \rvert^2 d \xi \bigg],
\end{align*}
\endgroup
where the limit holds in the sense of distributions and is independent of the choice of mollifier. Note that $\bigg[D_{6} (u) + \alpha^2 D_{7} (u) + \frac{1}{2} \alpha^2 D_{8} (u)\bigg]$ should be regarded as a single distribution, but we will write it as a sum of the defect terms $D_6, D_7$ and $D_8$ (which formally correspond to the limits of the different integrals) for later notational convenience.
\end{theorem}
\begin{proof}
We observe that $u^\epsilon, v^\epsilon \in L^\infty ((0,T); C^\infty (\mathbb{T}^3))$. Moreover, we get that $\partial_t v^\epsilon, \partial_t u^\epsilon \in L^\infty ((0,T); C^\infty (\mathbb{T}^3))$ by mollifying the weak formulation. Then by applying Lemma \ref{genclark} and taking the general test function $u^\epsilon \chi$, we get that
\begingroup
\allowdisplaybreaks
\begin{align*}
&\int_0^T \int_{\mathbb{T}^3 } u \cdot \partial_t (u^\epsilon \chi) dx dt + \alpha^2 \int_0^T \int_{\mathbb{T}^3 } \nabla u : \nabla \partial_t (u^\epsilon \chi) dx dt + \int_0^T \int_{\mathbb{T}^3 } (u \otimes u) : \nabla (u^\epsilon \chi) dx dt  \\
&+2\alpha^2 \int_0^T \int_{\mathbb{T}^3} \partial_k u \otimes \partial_k u : \nabla (u^\epsilon \chi) dx dt + \alpha^2 \int_0^T \int_{\mathbb{T}^3} u \otimes \partial_k u: \nabla \partial_k (u^\epsilon \chi) dx dt   \\
& + \alpha^2 \int_0^T \int_{\mathbb{T}^3} \partial_k u \otimes u : \nabla \partial_k  (u^\epsilon \chi) dx dt - \alpha^2 \int_0^T \int_{\mathbb{T}^3} \nabla u \cdot (\nabla u)^T : \nabla (u^\epsilon \chi) dx dt  \\
& + \int_0^T \int_{\mathbb{T}^3} p \nabla \cdot (u^\epsilon \chi) dx dt = 0.
\end{align*}
\endgroup
Then we mollify the equation of the Clark-$\alpha$ model, multiply it by $u \chi$ and subtract it from the previous expression, which gives
\begingroup
\allowdisplaybreaks
\begin{align*}
&\int_0^T \int_{\mathbb{T}^3 } u \cdot \partial_t (u^\epsilon \chi) dx dt - \int_0^T \int_{\mathbb{T}^3 } u \chi \cdot \partial_t u^\epsilon dx dt + \alpha^2 \int_0^T \int_{\mathbb{T}^3 } \nabla u : \nabla \partial_t (u^\epsilon \chi) dx dt \\
&- \alpha^2 \int_0^T \int_{\mathbb{T}^3 } \nabla (u \chi) : \nabla \partial_t u^\epsilon dx dt + \int_0^T \int_{\mathbb{T}^3 } (u \otimes u) : \nabla (u^\epsilon \chi) dx dt - \int_0^T \int_{\mathbb{T}^3} u \chi \cdot (\nabla \cdot (u \otimes u)^\epsilon) dx dt \\
&+2\alpha^2 \int_0^T \int_{\mathbb{T}^3} \partial_k u \otimes \partial_k u : \nabla (u^\epsilon \chi) dx dt - 2 \alpha^2 \int_0^T \int_{\mathbb{T}^3} u^\epsilon \chi \cdot (\nabla \cdot (\partial_k u \otimes \partial_k u)^\epsilon ) dx dt \\
&+ \alpha^2 \int_0^T \int_{\mathbb{T}^3} u \otimes \partial_k u: \nabla \partial_k (u^\epsilon \chi) dx dt - \alpha^2 \int_0^T \int_{\mathbb{T}^3} \partial_k (u \chi) \cdot (\nabla \cdot (u \otimes \partial_k u)^\epsilon ) dx dt  \\
& + \alpha^2 \int_0^T \int_{\mathbb{T}^3} \partial_k u \otimes u : \nabla \partial_k  (u^\epsilon \chi) dx dt - \alpha^2 \int_0^T \int_{\mathbb{T}^3} \partial_k (u \chi) \cdot (\nabla \cdot (\partial_k u \otimes u)^\epsilon) dx dt \\
&- \alpha^2 \int_0^T \int_{\mathbb{T}^3} \nabla u \cdot (\nabla u)^T : \nabla (u^\epsilon \chi) dx dt + \alpha^2 \int_0^T \int_{\mathbb{T}^3} u \chi \cdot (\nabla \cdot (\nabla u \cdot (\nabla u )^T )^\epsilon ) dx dt \\
& + \int_0^T \int_{\mathbb{T}^3} p \nabla \cdot (u^\epsilon \chi) dx dt - \int_0^T \int_{\mathbb{T}^3} u \chi \cdot \nabla p^\epsilon dx dt = 0.
\end{align*}
\endgroup
First we look at the time derivative parts, for which we have that
\begingroup
\allowdisplaybreaks
\begin{align*}
&\int_0^T \int_{\mathbb{T}^3 } u \partial_t (u^\epsilon \chi) dx dt - \int_0^T \int_{\mathbb{T}^3} u \chi \partial_t (u^\epsilon) dx dt + \alpha^2 \int_0^T \int_{\mathbb{T}^3 } \nabla u : \nabla \partial_t (u^\epsilon \chi) dx dt  \\
&- \alpha^2 \int_0^T \int_{\mathbb{T}^3} \nabla (u \chi) :  \partial_t \nabla u^\epsilon dx dt = - \bigg\langle \partial_t ( u \cdot u^\epsilon) + \alpha^2 \partial_t (\nabla u : \nabla u^\epsilon) - \alpha^2 \partial_t \partial_i \bigg( u_j^\epsilon \partial_i u_j + u_j \partial_i u_j^\epsilon \bigg) \\
&+ \alpha^2\partial_i \bigg( \partial_t u_j^\epsilon \partial_i u_j + \partial_t u_j \partial_i u_j^\epsilon \bigg) , \chi \bigg\rangle.
\end{align*}
\endgroup
We note that under the assumption that $u \in L^3 ((0,T); W^{1,9/4} (\mathbb{T}^3))$, we can prove that $p \in L^{3/2} ((0,T); L^{9/8} (\mathbb{T}^3))$ and $\partial_t v \in L^{3/2} ((0,T); W^{-1,9/8} (\mathbb{T}^3))$ and hence $\partial_t u \in L^{3/2} ((0,T); W^{1,9/8} (\mathbb{T}^3)) \subset L^3 ((0,T); L^{9/5} (\mathbb{T}^3))$. This means that $\partial_t u_j \partial_i u_j \in L^1 ((0,T); L^1 (\mathbb{T}^3))$. In the limit as $\epsilon \rightarrow 0$ this is the same as
\begin{equation*}
- \bigg\langle \partial_t ( \lvert u \rvert^2 ) + \alpha^2 \partial_t (\lvert \nabla u \rvert^2) + 2 \alpha^2 \partial_t \partial_i \bigg( u_j \partial_i u_j\bigg) + 2 \alpha^2\partial_i \bigg( \partial_t u_j \partial_i u_j \bigg), \chi \bigg\rangle.
\end{equation*}
The first part of the advective terms can be dealt with by introducing a defect term
\begin{equation*}
D_{6,\epsilon} (u) (x,t) \coloneqq \frac{1}{2} \int_{\mathbb{R}^3} \nabla_\xi \phi_\epsilon (\xi) \cdot \delta u(\xi;x,t) \lvert \delta u (\xi ;x,t) \lvert^2 d \xi,
\end{equation*}
As before we can write that
\begingroup
\allowdisplaybreaks
\begin{align*}
&\int_0^T \int_{\mathbb{T}^3 } (u \otimes u) : \nabla (u^\epsilon \chi) dx dt - \int_0^T \int_{\mathbb{T}^3 }  \chi u \cdot  (\nabla \cdot (u \otimes u)^\epsilon) dx dt \\
&= \int_0^T \int_{\mathbb{T}^3} \bigg[- \chi D_{6, \epsilon} (u) (x,t) - \frac{1}{2} (\lvert u \rvert^2 u)^\epsilon \cdot \nabla \chi + \frac{1}{2} u (\lvert u \rvert^2)^\epsilon \cdot \nabla \chi \bigg] dx dt \\
&+ \int_0^T \int_{\mathbb{T}^3 } (u \otimes u) : \nabla \chi \otimes u^\epsilon dx dt.
\end{align*}
\endgroup
Next we look at the terms coming from the advective term $\nabla \cdot (v \otimes u)$, we have that
\begingroup
\allowdisplaybreaks
\begin{align*}
&\alpha^2 \int_0^T \int_{\mathbb{T}^3} \partial_k u \otimes \partial_k u : \nabla (u^\epsilon \chi) dx dt -  \alpha^2 \int_0^T \int_{\mathbb{T}^3} \chi u^\epsilon \cdot (\nabla \cdot (\partial_k u \otimes \partial_k u)^\epsilon ) dx dt \\
& + \alpha^2 \int_0^T \int_{\mathbb{T}^3} \partial_k u \otimes u : \nabla \partial_k  (u^\epsilon \chi) dx dt - \alpha^2 \int_0^T \int_{\mathbb{T}^3} \partial_k (u^\epsilon \chi) \cdot (\nabla \cdot (\partial_k u \otimes u)^\epsilon) dx dt \\
&= \alpha^2 \int_0^T \int_{\mathbb{T}^3} \chi \bigg[ - D_{7,\epsilon} (u) - \partial_i ( u_j \partial_k u_j \partial_k u_i)^\epsilon + \partial_k u_i \partial_i (u_j \partial_k u_j)^\epsilon \bigg] dx dt \\
&+ \alpha^2 \int_0^T \int_{\mathbb{T}^3} \bigg[ \partial_k u \otimes \partial_k u : \nabla \chi \otimes u^\epsilon + \partial_k u \otimes u : \bigg( \nabla (u^\epsilon \partial_k \chi) + \nabla \chi \otimes \partial_k u^\epsilon \bigg) \\
&- \partial_k \chi u^\epsilon \cdot (\nabla \cdot (\partial_k u \otimes u)^\epsilon) \bigg]dx dt = - \alpha^2 \bigg\langle D_{7,\epsilon} (u) - \nabla \cdot (u_j \partial_k u_j \partial_k u)^\epsilon + \nabla \cdot (\partial_k u (u_j \partial_k u_j)^\epsilon ), \chi \bigg\rangle \\
&+ \alpha^2 \int_0^T \int_{\mathbb{T}^3} \bigg[ \partial_k u \otimes \partial_k u : \nabla \chi \otimes u^\epsilon + \partial_k u \otimes u : \bigg( \nabla (u^\epsilon \partial_k \chi) + \nabla \chi \otimes \partial_k u^\epsilon \bigg) \\
&- \partial_k \chi u^\epsilon \cdot (\nabla \cdot (\partial_k u \otimes u)^\epsilon) \bigg]dx dt.
\end{align*}
\endgroup
In the above we have introduced the following defect term
\begin{equation*}
D_{7,\epsilon} (u) (x,t) = \int_{\mathbb{R}^3} \partial_i \phi_\epsilon (\xi) \delta u_j (\xi;x,t) \delta \partial_k u_i (\xi;x,t) \delta \partial_k u_j (\xi;x,t) d \xi.
\end{equation*}
Finally, we are left with the remaining terms
\begingroup
\allowdisplaybreaks
\begin{align*}
&\alpha^2 \int_0^T \int_{\mathbb{T}^3} \partial_k u_i \partial_k u_j  \partial_i (u^\epsilon_j \chi) dx dt - \alpha^2 \int_0^T \int_{\mathbb{T}^3} u_j \chi  \partial_i (\partial_k u_i  \partial_k u_j)^\epsilon dx dt \\
&+ \alpha^2 \int_0^T \int_{\mathbb{T}^3} u_i \partial_k u_j \partial_i \partial_k (u^\epsilon_j \chi) dx dt - \alpha^2 \int_0^T \int_{\mathbb{T}^3} \partial_k (u_j \chi)  \partial_i  (u_i  \partial_k u_j)^\epsilon  dx dt  \\
&- \alpha^2 \int_0^T \int_{\mathbb{T}^3} \partial_k u_i  \partial_k u_j \partial_i (u^\epsilon_j \chi) dx dt + \alpha^2 \int_0^T \int_{\mathbb{T}^3} u_j \chi  \partial_i (\partial_k u_i  \partial_k u_j )^\epsilon  dx dt \\
&= \alpha^2 \int_0^T \int_{\mathbb{T}^3} \bigg[ u_i \partial_k u_j \partial_i  (u^\epsilon_j \partial_k \chi) + u_i \partial_k u_j \partial_k  u^\epsilon_j \partial_i \chi + \partial_k \partial_i  \chi u_j   (u_i  \partial_k u_j)^\epsilon + \partial_k  \chi \partial_i u_j  (u_i  \partial_k u_j)^\epsilon \bigg] dx dt \\
&+ \alpha^2 \int_0^T \int_{\mathbb{T}^3} \chi \bigg[ - \frac{1}{2} D_{8,\epsilon} (u)  - \frac{1}{2} \partial_i (u_i \partial_k u_j \partial_k u_j )^\epsilon + \frac{1}{2} u_i \partial_i (\partial_k u_j \partial_k u_j)^\epsilon \bigg]dx dt.
\end{align*}
\endgroup
In the above we have introduced the defect term given by
\begin{align*}
D_{8,\epsilon} (x,t) &= \int_{\mathbb{R}^3} \partial_i \phi_\epsilon (\xi) \delta u_i (\xi ; x,t) \lvert \delta \partial_k u_j (\xi ;x,t) \rvert^2 d \xi = - \partial_i (u_i \partial_k u_j \partial_k u_j)^\epsilon \\
&+ u_i \partial_i (\partial_k u_j \partial_k u_j)^\epsilon + 2 \partial_k u_j \partial_i (u_i \partial_k u_j)^\epsilon - 2u_i \partial_k u_j \partial_i \partial_k u_j^\epsilon. \nonumber
\end{align*}
One can write the pressure terms as
\begin{align*}
&\int_0^T \int_{\mathbb{T}^3} p \nabla \cdot (u^\epsilon \chi) dx dt - \int_0^T \int_{\mathbb{T}^3} u \chi \cdot \nabla p^\epsilon dx dt = - \bigg\langle \nabla \cdot (p u^\epsilon + p^\epsilon u ), \chi \bigg\rangle.
\end{align*}
Finally, this allows us to write down the local equation of energy
\begingroup
\allowdisplaybreaks
\begin{align*}
&\bigg\langle \partial_t ( u \cdot u^\epsilon) + \alpha^2 \partial_t (\nabla u : \nabla u^\epsilon) - \alpha^2 \partial_t \partial_i \bigg( u_j^\epsilon \partial_i u_j + u_j \partial_i u_j^\epsilon \bigg) + \alpha^2\partial_i \bigg( \partial_t u_j^\epsilon \partial_i u_j + \partial_t u_j \partial_i u_j^\epsilon \bigg) \\
&+ D_{6,\epsilon} (u) + \frac{1}{2} \nabla \cdot \big( (\lvert u \rvert^2)^\epsilon u - (\lvert u \rvert^2 u)^\epsilon \big) + \nabla \cdot ( \lvert u \rvert^2 u^\epsilon) + \alpha^2 D_{7,\epsilon} (u) - \alpha^2 \nabla \cdot (u_j \partial_k u_j \partial_k u)^\epsilon \\
&+ \alpha^2 \nabla \cdot (\partial_k u (u_j \partial_k u_j)^\epsilon ) + \alpha^2 \partial_i ( u^\epsilon_j \partial_k u_j \partial_k u_i) + \alpha^2 \partial_i \partial_k (u^\epsilon_j \partial_k u_i u_j) + \alpha^2 \partial_k (\partial_i u^\epsilon_j  \partial_k u_i u_j) \\
&+ \alpha^2 \partial_k  (\partial_i u^\epsilon_j \partial_k u_i u_j) + \alpha^2 \partial_i \partial_k (u_j  (\partial_k u_i u_j)^\epsilon ) + \alpha^2 \partial_k (\partial_iu_j  (\partial_k u_i u_j)^\epsilon ) + \frac{1}{2} \alpha^2 D_{8,\epsilon} (u) \\
&+ \frac{1}{2} \alpha^2 \partial_i (u_i \partial_k u_j \partial_k u_j)^\epsilon - \frac{1}{2} \alpha^2 \partial_i (u_i (\partial_k u_j \partial_k u_j)^\epsilon )+ \nabla \cdot (p u^\epsilon + p^\epsilon u ), \chi \bigg\rangle = 0.
\end{align*}
\endgroup
As before, we know that $D_{6,\epsilon} (u), D_{7,\epsilon} (u) $ and $D_{8,\epsilon} (u)$ all make sense for fixed $\epsilon > 0$ as an element in $L^1 ((0,T); L^1 (\mathbb{T}^3))$ due to the assumption that $u \in L^3 ((0,T); W^{1,9/4} (\mathbb{T}^3))$. We can write the above equation as an equation holding in the sense of distributions for $[D_{6,\epsilon} (u) + \alpha^2 D_{7,\epsilon} (u) + \frac{1}{2} \alpha^2 D_{8,\epsilon} (u)]$. Now we observe that all the other terms in the equation converge in the sense of distributions as $\epsilon \rightarrow 0$ and therefore by arguing as before $\lim_{\epsilon \rightarrow 0} \big[ D_{6,\epsilon} (u) + \alpha^2 D_{7,\epsilon} (u) + \frac{1}{2} \alpha^2 D_{8,\epsilon} (u) \big]$ makes sense as a distribution. We will denote this limit, abusing notation, by $[D_6 (u) + \alpha^2 D_7 (u) + \frac{1}{2} \alpha^2 D_8 (u)]$, even though we have not shown at this stage that each individual defect term is converging. Later, we will show that under more restricted condition that each defect term converges to 0. In particular, from above we arrive at the following equation of local energy balance
\begingroup
\allowdisplaybreaks
\begin{align*}
&\bigg\langle \partial_t (\lvert u \rvert^2) + \alpha^2 \partial_t (\lvert \nabla u \rvert^2) - 2 \alpha^2 \partial_t \partial_i \bigg( u_j \partial_i u_j \bigg) + 2 \alpha^2\partial_i \bigg( \partial_t u_j \partial_i u_j  \bigg) + \nabla \cdot ( \lvert u \rvert^2 u) \\
&  + 2 \alpha^2 \partial_i ( u_j \partial_k u_j \partial_k u_i) + \alpha^2 \partial_i \partial_k (u_j \partial_k u_i u_j) + 3\alpha^2 \partial_k (\partial_i u_j  \partial_k u_i u_j) \\
& + \bigg[D_{6} (u) + \alpha^2 D_{7} (u) + \frac{1}{2} \alpha^2 D_{8} (u) \bigg]  + 2\nabla \cdot (p u ), \chi \bigg\rangle = 0.
\end{align*}
\endgroup
\end{proof}
We now state the usual sufficient condition for the defect terms to be zero.
\begin{proposition} \label{zerodefectclark}
Let $u$ be a weak solution of the inviscid Clark-$\alpha$ model such that
\begin{align*}
&\int_{\mathbb{T}^3} \lvert \delta u (\xi ;x,t) \rvert^3 dx \leq C(t) \lvert \xi \rvert \sigma_6 (\lvert \xi \rvert), \\
&\int_{\mathbb{T}^3} \lvert \delta u (\xi ;x,t) \rvert \lvert \delta \nabla u (\xi;x,t) \rvert^2 dx \leq C(t) \lvert \xi \rvert \sigma_7 (\lvert \xi \rvert),
\end{align*}
where we assume that $C \in L^1 (0,T)$, while $\sigma_j \in L^\infty_{\mathrm{loc}} (\mathbb{R})$ such that $\sigma_j (\lvert \xi \rvert)\rightarrow 0$ as $\lvert \xi \rvert \rightarrow 0$, for $j=6,7$. Then  $\lim_{\epsilon \to 0} D_{k,\epsilon}(u)=D_k (u)=0$, for $k=6,7,8$.
\end{proposition}
\begin{proof}
The proof proceeds in the same fashion as the proof of Proposition \ref{zerodefect}.
\end{proof}
We now specify a sufficient regularity assumption for the sufficient conditions in Proposition \ref{zerodefectclark} to be satisfied.
\begin{proposition}
Let $u$ be a weak solution of the inviscid Clark-$\alpha$ model with $u \in L^3 ((0,T); \linebreak B^s_{3,\infty} (\mathbb{T}^3))$ with $s > 1$, then $\lim_{\epsilon \rightarrow 0} D_{k,\epsilon} (u) = D_k (u) = 0$ for $k = 6,7,8$. This implies that the weak solution conserves energy.
\end{proposition}
\begin{proof}
To prove the first inequality in Proposition \ref{zerodefectclark} we observe that the Besov embedding theorem gives us that $H^1 (\mathbb{T}^3) \subset B^{1/2}_{3,\infty} (\mathbb{T}^3)$ and hence we get that
\begin{equation*}
\int_{\mathbb{T}^3} \lvert \delta u (\xi;x,t) \rvert^3 dx \leq \lvert \xi \rvert^{3/2} \lVert u \rVert_{H^1 }^3.
\end{equation*}
One can observe that the second inequality in Proposition \ref{zerodefectclark} holds because
\begin{equation*}
\int_{\mathbb{T}^3} \lvert \delta u (\xi;x,t) \rvert \lvert \delta \nabla u (\xi;x,t) \rvert^2 dx \leq \lvert \xi \rvert^{3s - 2} \lVert u \rVert_{B^s_{3,\infty}}^3.
\end{equation*}
This follows from using property \eqref{besovinequality} for functions in Besov spaces.
Therefore the assumptions of Proposition \ref{zerodefectclark} are satisfied and we conclude that $D_6 (u) = D_7 (u) = D_8 (u) = 0$.
\end{proof}
This proposition allows us to prove conservation of energy.
\begin{theorem}
Let $u$ be a weak solution of the inviscid Clark-$\alpha$ model with $u \in L^3 ((0,T); B^s_{3,\infty} (\mathbb{T}^3))$ with $s > 1$, then the solution conserves energy. In particular, it holds that
\begin{equation*}
\lVert u (\cdot, t_1) \rVert_{H^1} = \lVert u (\cdot, t_2) \rVert_{H^1},
\end{equation*}
for almost every $t_1, t_2 \in (0,T)$.
\end{theorem}
\section{The Leray-$\alpha$ MHD model} \label{lerayMHDsection}
The equations for the Leray-$\alpha$ MHD model are
\begingroup
\allowdisplaybreaks
\begin{align*}
&\partial_t v + (u \cdot \nabla) v + \nabla p + \frac{1}{2} \nabla \lvert B \rvert^2 = (B \cdot \nabla) B, \\
&\partial_t B + (u \cdot \nabla) B - (B \cdot \nabla )v = 0, \\
&v = u - \alpha^2 \Delta u, \quad \nabla \cdot u = \nabla \cdot v = \nabla \cdot B = 0.
\end{align*}
\endgroup
\begin{definition}
A weak solution of the Leray-$\alpha$ MHD model is a triplet $v \in L^\infty ( (0,T); L^2 (\mathbb{T}^3)), B \in L^\infty ((0,T); L^2 (\mathbb{T}^3))$ and $p \in L^\infty ((0,T); L^1 (\mathbb{T}^3)$ if for all $\psi \in \mathcal{D } (\mathbb{T}^3 \times (0,T); \mathbb{R}^3)$ and $\chi \in \mathcal{D} (\mathbb{T}^3 \times (0,T); \mathbb{R})$ the following equations hold
\begingroup
\allowdisplaybreaks
\begin{align*}
&\int_0^T \int_{\mathbb{T}^3} v \cdot \partial_t \psi dx dt + \int_0^T \int_{\mathbb{T}^3} v_j u_i \partial_i \psi_j dx dt + \int_0^T \int_{\mathbb{T}^3} p \nabla \cdot \psi dx dt + \frac{1}{2} \int_0^T \int_{\mathbb{T}^3} \lvert B \rvert^2 \nabla \cdot \psi dx dt \\
&- \int_0^T \int_{\mathbb{T}^3} (B \otimes B): \nabla \psi dx dt = 0,  \\
&\int_0^T \int_{\mathbb{T}^3} B \cdot \partial_t \psi dx dt + \int_0^T \int_{\mathbb{T}^3} B_j u_i \partial_i \psi_j dx dt - \int_0^T \int_{\mathbb{T}^3} v_j B_i \partial_i \psi_j dx dt = 0, \\
&\int_0^T \int_{\mathbb{T}^3} u_i \partial_i \chi dx dt = \int_0^T \int_{\mathbb{T}^3} B_i \partial_i \chi dx dt = 0.
\end{align*}
\endgroup
We fix the constant in the definition of the pressure by demanding that
\begin{equation*}
\int_{\mathbb{T}^3} p(x,t) dx = 0.
\end{equation*}
\end{definition}
Once again, we extend  the space of test functions.
\begin{lemma} \label{gentestMHD}
The weak formulation of the Leray-$\alpha$ MHD model still holds for $\psi \in W^{1,1}_0 ( (0,T); \linebreak L^2 (\mathbb{T}^3)) \cap L^1 ((0,T); H^3 (\mathbb{T}^3))$.
\end{lemma}
\begin{proof}
The proof works the same way as for the other models.
\end{proof}
Now we establish the equation of local energy balance.
\begin{theorem}
Let $(v,B)$ be a weak solution of the Leray-$\alpha$ MHD model such that $v,B \in L^3 ((0,T); L^3 (\mathbb{T}^3))$. Then the following equation of local energy balance holds in the sense of distributions
\begin{align*}
&\partial_t (\lvert v \rvert^2 + \lvert B \rvert^2) + 2 \nabla \cdot (pv) +  \nabla \cdot (\lvert v \rvert^2 u) + 2 \nabla \cdot (\lvert B \rvert^2 v) + \nabla \cdot (\lvert B \rvert^2 u) -2 \nabla \cdot ( (v\cdot B) B) \\
&+ \big[ D_{9} (u,v) + D_{10} (u,B) - D_{11} (v,B) \big] = 0.
\end{align*}
Here the sum of the three defect terms is given by the limit
\begingroup
\allowdisplaybreaks
\begin{align*}
&\big[ D_{9} (u,v) + D_{10} (u,B) - D_{11} (v,B) \big] (x,t) \coloneqq \lim_{\epsilon \rightarrow 0} \bigg[ \frac{1}{2} \int_{\mathbb{R}^3} \nabla_\xi \phi_\epsilon (\xi) \cdot \delta u (\xi ;x,t) \lvert \delta v (\xi; x,t ) \rvert^2 d \xi \\
&+ \frac{1}{2} \int_{\mathbb{R}^3} \nabla_\xi \phi_\epsilon (\xi) \cdot \delta u (\xi ;x,t) \lvert \delta B (\xi; x,t ) \rvert^2 d \xi -  \int_{\mathbb{R}^3} \nabla_\xi \phi_\epsilon (\xi) \cdot \delta B (\xi ;x,t) ( \delta B (\xi ;x,t) \cdot \delta v (\xi ;x,t)) d \xi \bigg].
\end{align*}
\endgroup
Note that the limit holds in the sense of distributions and is independent of the choice of mollifier. Like in the case of the Clark-$\alpha$ model, $\big[ D_{9} (u,v) + D_{10} (u,B) - D_{11} (v,B) \big]$ should be regarded as a single distribution.
\end{theorem}
\begin{proof}
Mollifying the equations of the Leray-$\alpha$ MHD model in space (with $\phi_\epsilon$) gives that
\begin{align}
&\partial_t v^\epsilon + \nabla \cdot (u \otimes v)^\epsilon + \nabla p^\epsilon + \frac{1}{2} \nabla (\lvert B \rvert^2)^\epsilon - \nabla \cdot (B \otimes B)^\epsilon = 0, \label{mollifiedMHD1}\\
&\partial_t B^\epsilon + \nabla \cdot (u \otimes B)^\epsilon - \nabla \cdot (B \otimes v)^\epsilon = 0. \label{mollifiedMHD2}
\end{align}
It follows that $v^\epsilon, B^\epsilon \in L^\infty ((0,T); C^\infty (\mathbb{T}^3))$, as well as that $\partial_t v^\epsilon, \partial_t B^\epsilon \in L^\infty ((0,T); C^\infty (\mathbb{T}^3))$. As a result we can conclude that $v^\epsilon, B^\epsilon \in W^{1,\infty} ((0,T); C^\infty (\mathbb{T}^3))$. This allows us to apply Lemma \ref{gentestMHD} and use $v^\epsilon \chi$ and $B^\epsilon \chi$ as test functions in the weak formulations, where $\chi \in \mathcal{D} (\mathbb{T}^3 \times (0,T); \mathbb{R})$. Subtracting the mollified equations multiplied by $v \chi$ and $B \chi$ we obtain that
\begingroup
\allowdisplaybreaks
\begin{align*}
&\int_0^T \int_{\mathbb{T}^3} \bigg[ v \cdot \partial_t (v^\epsilon \chi) - v \chi \cdot \partial_t v^\epsilon + B \cdot \partial_t (B^\epsilon \chi) - B \chi \cdot \partial_t B^\epsilon  + u \otimes v : \nabla (v^\epsilon \chi) - v \cdot (\nabla \cdot (u \otimes v)^\epsilon) \\
&-  \chi v \cdot \nabla p^\epsilon + p \partial_i (v^\epsilon_i \chi)  + \frac{1}{2}  \lvert B \rvert^2 \nabla \cdot ( v^\epsilon \chi) - \frac{1}{2} v \chi \nabla (\lvert B \rvert^2)^\epsilon - B \otimes B : \nabla (v^\epsilon \chi) + v \cdot (\nabla \cdot (B \otimes B)^\epsilon) \\
&+   u \otimes B : \nabla (\chi B^\epsilon) - \chi B \cdot (\nabla \cdot (u \otimes B)^\epsilon ) - B \otimes v : \nabla (\chi B^\epsilon) + B \chi \cdot (\nabla \cdot ( B \otimes v)^\epsilon) \bigg] dx dt = 0.
\end{align*}
\endgroup
The time derivative terms then become
\begin{align*}
&\int_0^T \int_{\mathbb{T}^3} \bigg[ v \cdot \partial_t (v^\epsilon \chi) - v \chi \cdot \partial_t v^\epsilon +  B \cdot \partial_t (B^\epsilon \chi) - B \chi \cdot \partial_t B^\epsilon \bigg] dx dt \\
&= \int_0^T \int_{\mathbb{T}^3} \big[ v \cdot v^\epsilon + B \cdot B^\epsilon \bigg] \partial_t \chi dx dt = - \bigg\langle \partial_t \big( v \cdot v^\epsilon + B \cdot B^\epsilon \big), \chi \bigg\rangle.
\end{align*}
For the pressure terms we get that
\begin{align*}
\int_0^T \int_{\mathbb{T}^3} \bigg[- \chi v \cdot \nabla p^\epsilon + p \partial_i (v^\epsilon_i \chi) \bigg] dx dt &=  \int_0^T \int_{\mathbb{T}^3} \bigg[ p^\epsilon v_i \partial_i \chi + p v^\epsilon_i \partial_i  \chi \bigg] dx dt \\
&= -\langle \nabla \cdot (p^\epsilon v + p v^\epsilon), \chi \rangle.
\end{align*}
Note that $p \in L^{3/2} ((0,T); L^{3/2} (\mathbb{T}^3))$ and therefore $p^\epsilon v + p v^\epsilon \xrightarrow[]{\epsilon \rightarrow 0} 2 p v$ in $L^{1} ((0,T); L^1 (\mathbb{T}^3))$. Several of the advective terms can be written down straightforwardly (analogously to the Leray-$\alpha$ model)
\begin{align*}
&\int_0^T \int_{\mathbb{T}^3} \bigg[ u \otimes v : \nabla (v^\epsilon \chi)  - v \cdot (\nabla \cdot (u \otimes v)^\epsilon) \bigg] dx dt \\
&= -\int_0^T \int_{\mathbb{T}^3} \bigg[\chi D_{9,\epsilon} (u,v) - \frac{1}{2}\bigg(( v_j v_j)^\epsilon u_i  -(u_i v_j v_j)^\epsilon \bigg) \partial_i \chi  - v_j u_i v_j^\epsilon \partial_i \chi  \bigg] dx dt \\
&= -\bigg\langle D_{9,\epsilon} (u,v) + \frac{1}{2} \nabla \cdot ( (\lvert v \rvert^2)^\epsilon u -  (\lvert v \rvert^2 u)^\epsilon) + \nabla \cdot ( (v \cdot v^\epsilon) u) , \chi \bigg\rangle,
\end{align*}
where we have introduced the defect term
\begin{equation*}
D_{9,\epsilon} (u,v) (x,t) \coloneqq \frac{1}{2} \int_{\mathbb{R}^3} \nabla_\xi \phi_\epsilon (\xi) \cdot \delta u (\xi ;x,t) \lvert \delta v (\xi; x,t ) \rvert^2 d \xi,
\end{equation*}
Now we can write that
\begin{align*}
\frac{1}{2} \int_0^T \int_{\mathbb{T}^3} \bigg[ \lvert B \rvert^2 \nabla \cdot ( v^\epsilon \chi) - v \chi \nabla (\lvert B \rvert^2)^\epsilon \bigg] dx dt &= \frac{1}{2} \int_0^T \int_{\mathbb{T}^3} \bigg[ \lvert B \rvert^2 v^\epsilon \cdot \nabla \chi + v (\lvert B \rvert^2)^\epsilon \cdot \nabla \chi \bigg] dx dt \\
&= -\bigg\langle \nabla \cdot \big( \lvert B \rvert^2 v^\epsilon + (\lvert B \rvert^2)^\epsilon v \big) , \chi \bigg\rangle = 0.
\end{align*}
For the other advective terms we find that
\begin{align*}
&\int_0^T \int_{\mathbb{T}^3} \bigg[ u \otimes B : \nabla (\chi B^\epsilon) - \chi B \cdot (\nabla \cdot (u \otimes B)^\epsilon ) \bigg] dx dt \\
&= -\int_0^T \int_{\mathbb{T}^3} \bigg[\chi D_{10,\epsilon} (u,B) - \frac{1}{2}\bigg(( B_j B_j)^\epsilon u_i  -(u_i B_j B_j)^\epsilon \bigg) \partial_i \chi   - B_j u_i B_j^\epsilon \partial_i \chi  \bigg] dx dt \\
&= -\bigg\langle D_{10,\epsilon} (u,B) + \frac{1}{2} \nabla \cdot ( (\lvert B \rvert^2)^\epsilon u - (\lvert B \rvert^2 u)^\epsilon ) + \nabla \cdot ( (B \cdot B^\epsilon) u) , \chi \bigg\rangle,
\end{align*}
where we have introduced the defect term
\begin{equation*}
D_{10,\epsilon} (u,B) (x,t) \coloneqq \frac{1}{2} \int_{\mathbb{R}^3} \nabla_\xi \phi_\epsilon (\xi) \cdot \delta u (\xi ;x,t) \lvert \delta B (\xi; x,t ) \rvert^2 d \xi,
\end{equation*}
Finally, we are left with
\begin{align*}
\int_0^T \int_{\mathbb{T}^3} \bigg[v \cdot (\nabla \cdot (B \otimes B)^\epsilon) - B \otimes B : \nabla (v^\epsilon \chi) - B \otimes v : \nabla (\chi B^\epsilon) + B \chi \cdot (\nabla \cdot ( B \otimes v)^\epsilon) \bigg] dx dt.
\end{align*}
We introduce another defect term, which can be calculated to be
\begin{align*}
D_{11,\epsilon} (v,B) &\coloneqq \int_{\mathbb{R}^3} \nabla_\xi \phi_\epsilon (\xi) \cdot \delta B (\xi ;x,t) ( \delta B (\xi ;x,t) \cdot \delta v (\xi ;x,t)) d \xi \\
&= - \partial_i (B_i B_j v_j)^\epsilon + B_i \partial_i (B_j v_j)^\epsilon + B_j \partial_i (B_i v_j)^\epsilon + v_j \partial_i (B_i B_j)^\epsilon - B_i v_j \partial_i B_j^\epsilon - B_i B_j \partial_i v_j^\epsilon.
\end{align*}
Therefore we can write the different terms as follows
\begingroup
\allowdisplaybreaks
\begin{align*}
&\int_0^T \int_{\mathbb{T}^3} \bigg[v \cdot (\nabla \cdot (B \otimes B)^\epsilon) - B \otimes B : \nabla (v^\epsilon \chi) - B \otimes v : \nabla (\chi B^\epsilon) + B \chi \cdot (\nabla \cdot ( B \otimes v)^\epsilon) \bigg] dx dt \\
&= \int_0^T \int_{\mathbb{T}^3} \bigg[ - B \otimes B : \nabla \chi \otimes v^\epsilon - B \otimes v : \nabla \chi \otimes B^\epsilon + \chi D_{11,\epsilon} (v,B) + \chi \nabla \cdot ( (B \cdot v) B)^\epsilon \\
&- \chi B \cdot \nabla ( B \cdot v)^\epsilon \bigg] dx dt = \bigg\langle \nabla \cdot ( (v^\epsilon \cdot B) B) + \nabla \cdot ( (B^\epsilon \cdot v) B) +  D_{11,\epsilon} (v,B) +  \nabla \cdot ( (B \cdot v) B)^\epsilon \\
&-  \nabla \cdot ( ( B \cdot v)^\epsilon B), \chi \bigg\rangle.
\end{align*}
\endgroup
From this we can conclude that we have the following equation of local energy balance
\begin{align*}
\bigg\langle &-\partial_t \big( v \cdot v^\epsilon + B \cdot B^\epsilon \big) - \nabla \cdot (p^\epsilon v + p v^\epsilon) - D_{9,\epsilon} (u,v) - \frac{1}{2} \nabla \cdot ( (\lvert v \rvert^2)^\epsilon u -  (\lvert v \rvert^2 u)^\epsilon ) \\
&- \nabla \cdot ( (v \cdot v^\epsilon) u) -\nabla \cdot \big( \lvert B \rvert^2 v^\epsilon + (\lvert B \rvert^2)^\epsilon v \big) - D_{10,\epsilon} (u,B) - \frac{1}{2} \nabla \cdot ( (\lvert B \rvert^2)^\epsilon u -  (\lvert B \rvert^2 u)^\epsilon ) \\
&- \nabla \cdot ( (B \cdot B^\epsilon) u) + \nabla \cdot ( (v^\epsilon \cdot B) B) + \nabla \cdot ( (B^\epsilon \cdot v) B) +  D_{11,\epsilon} (v,B) +  \nabla \cdot ( (B \cdot v) B)^\epsilon \\
&-  \nabla \cdot (( B \cdot v)^\epsilon B), \chi \bigg\rangle = 0.
\end{align*}
Now we observe that for fixed $\epsilon > 0$, the defect terms $D_{9,\epsilon}, D_{10,\epsilon}$ and $D_{11,\epsilon}$ are all individually well-defined as elements in $L^1 ((0,T); L^1 (\mathbb{T}^3))$. Moreover, we can write the following equation for the sum of the defect terms
\begingroup
\allowdisplaybreaks
\begin{align*}
&-D_{9,\epsilon} (u,v) - D_{10,\epsilon} (u,B) + D_{11,\epsilon} (v,B) = \partial_t \big( v \cdot v^\epsilon + B \cdot B^\epsilon \big) + \nabla \cdot (p^\epsilon v + p v^\epsilon) \\
&+ \frac{1}{2} \nabla \cdot ( (\lvert v \rvert^2)^\epsilon u -  (\lvert v \rvert^2 u)^\epsilon ) + \nabla \cdot ( (v \cdot v^\epsilon) u) +\nabla \cdot \big( \lvert B \rvert^2 v^\epsilon + (\lvert B \rvert^2)^\epsilon v \big) + \frac{1}{2} \nabla \cdot ( (\lvert B \rvert^2)^\epsilon u \\
&-  (\lvert B \rvert^2 u)^\epsilon ) + \nabla \cdot ( (B \cdot B^\epsilon) u) - \nabla \cdot ( (v^\epsilon \cdot B) B) - \nabla \cdot ( (B^\epsilon \cdot v) B) -  \nabla \cdot ( (B \cdot v) B)^\epsilon \\
&+  B \cdot \nabla ( B \cdot v)^\epsilon,
\end{align*}
\endgroup
which holds in the sense of distributions. Now we observe that the right-hand converges in the sense of distributions to
\begin{align*}
\partial_t (\lvert v \rvert^2 + \lvert B \rvert^2) + 2 \nabla \cdot (pv) + \nabla \cdot (\lvert v \rvert^2 u) + 2 \nabla \cdot (\lvert B \rvert^2 v) + \nabla \cdot (\lvert B \rvert^2 u) -2 \nabla \cdot ( (v\cdot B) B).
\end{align*}
The limit exists as a distribution and is independent of the choice of mollifier, therefore $\lim_{\epsilon \rightarrow 0} \bigg[D_{9,\epsilon} (u,v) + D_{10,\epsilon} (u,B) - D_{11,\epsilon} (v,B) \bigg]$ exists as a distribution and the equation of local energy balance holds in the sense of distributions.
\end{proof}
Now we state sufficient conditions for the defect term to be zero.
\begin{proposition} \label{zerodefectMHD}
Assume that $(v,B) \in L^3 ((0,T); L^3 (\mathbb{T}^3))$ is a weak solution of the inviscid and irresistive Leray-$\alpha$ MHD model. Morever, we assume that
\begin{align*}
&\int_{\mathbb{T}^3} \lvert \delta u (\xi; x,t) \rvert \lvert \delta v (\xi; x,t) \rvert^2 dx \leq C(t) \lvert \xi \rvert \sigma_9 (\lvert \xi \rvert), \\
&\int_{\mathbb{T}^3} \lvert \delta u (\xi; x,t) \rvert \lvert \delta B (\xi; x,t) \rvert^2 dx \leq C(t) \lvert \xi \rvert \sigma_{10} (\lvert \xi \rvert), \\
&\int_{\mathbb{T}^3} \lvert \delta v (\xi; x,t) \rvert \lvert \delta B (\xi; x,t) \rvert^2 dx \leq C(t) \lvert \xi \rvert \sigma_{11} (\lvert \xi \rvert),
\end{align*}
where $C \in L^1 (0,T)$ and $\sigma_j \in L^\infty_{\mathrm{loc}} (\mathbb{R})$ such that $\sigma_j (\lvert \xi \rvert) \rightarrow 0$ as $\lvert \xi \rvert \rightarrow 0$, for $j=9,10,11$. Then $\lim_{\epsilon \rightarrow 0} D_{k,\epsilon} (u,v,B) = D_k (u,v,B) = 0$ for $k=9,10,11$. This then implies that the weak solution conserves energy.
\end{proposition}
\begin{proof}
The proof is completely analogous to the proof of Proposition \ref{zerodefect}.
\end{proof}
We now prove a proposition which states that under given regularity assumptions these conditions are met.
\begin{proposition}
Let $(v,B)$ be a weak solution of the inviscid and irresistive Leray-$\alpha$ MHD model such that $v \in L^3 ((0,T); B^s_{3,\infty} (\mathbb{T}^3))$ and $B \in L^3 ((0,T); B^r_{3,\infty} (\mathbb{T}^3))$ with $s, r > 0$ and $s + 2 r > 1$. Then $\lim_{\epsilon \rightarrow 0} D_{k,\epsilon} (u,v,B) = D_k (u,v,B) = 0$ for $k=9,10,11$ (which implies conservation of energy).
\end{proposition}
\begin{proof}
The proof for $D_{9}$ and $D_{10}$ is analogous to the case for the Leray-$\alpha$ model, because $v(\cdot, t) \in B^s_{3,\infty} (\mathbb{T}^3)$ implies that $u(\cdot, t) \in B^{2 + s}_{3,\infty} (\mathbb{T}^3)$. For $D_{11} (u,v)$ we observe that
\begin{equation*}
\int_{\mathbb{T}^3} \lvert \delta v (\xi; x,t) \rvert \lvert \delta B (\xi; x,t) \rvert^2 dx \leq \lvert \xi \rvert^{s + 2 r} \lVert v \rVert_{B^s_{3,\infty}} \lVert B \rVert_{B^r_{3,\infty}}^2.
\end{equation*}
To obtain this bound we have used inequality \eqref{besovinequality}.
Then we can take $\sigma_{11} (\lvert \xi \rvert) \coloneqq \lvert \xi \rvert^{s + 2 r - 1}$, which satisfies the conditions of Proposition \ref{zerodefectMHD}. Therefore it follows from the proposition that $D_{11} (v,B) = 0$.
\end{proof}
Finally, we can prove that the weak solution conserves energy.
\begin{theorem}
Let $(v,B)$ be a weak solution of the inviscid and irresistive Leray-$\alpha$ MHD model, then if $v \in L^3 ((0,T); B^s_{3,\infty} (\mathbb{T}^3))$ and $B \in L^3 ((0,T); B^r_{3,\infty} (\mathbb{T}^3))$ with $s, r > 0$ and $s + 2 r > 1$ the weak solution conserves energy. So for almost all $t_1 ,t_2 \in (0,T)$ it holds that
\begin{equation*}
\lVert v(t_1, \cdot) \rVert_{L^2}^2 + \lVert B(t_1, \cdot) \rVert_{L^2}^2 = \lVert v(t_2, \cdot) \rVert_{L^2}^2 + \lVert B(t_2, \cdot) \rVert_{L^2}^2.
\end{equation*}
\end{theorem}
\begin{proof}
The proof works the same way as the proof of Theorem \ref{conservationtheorem}.
\end{proof}
\end{appendices}
	
\bibliographystyle{abbrv}

\bibliography{subgrid_scale_models_final_version}

\end{document}